\documentclass[11pt]{amsart}
\usepackage{geometry}                
\usepackage{amsmath}
\usepackage{amscd}
\usepackage{amsthm}
\usepackage{graphicx}
\usepackage{amssymb}
\usepackage{epstopdf}
\usepackage{fancyhdr}
\theoremstyle{plain}
\newtheorem{Thm}{Theorem}[section]

\newtheorem{Prop}[Thm]{Proposition}
\newtheorem{Cor}[Thm]{Corollary}
\newtheorem{Lem}[Thm]{Lemma}

\theoremstyle{definition}

\newtheorem{Expl}[Thm]{Example}
\newtheorem{Rem}[Thm]{Remark}

\numberwithin{equation}{section}

\title{Semi-orthogonal decomposition of a derived category of a $3$-fold with an ordinary double point \\
Dedicated to Professor Miles Reid for his seventieth birthday}
\author{Yujiro Kawamata}
\begin{document}
\maketitle

\begin{abstract}
We consider semi-orthogonal decompositions of derived categories 
for $3$-dimensional projective varieties
in the case when the varieties have ordinary double points.
\end{abstract}

14F05, 14F17, 14E05

\tableofcontents

\section{Introduction}

The bounded derived category of coherent sheaves $D^b(\text{coh}(X))$ 
of an algebraic variety $X$ reflects
important properties of the variety $X$.
The abelian category of coherent sheaves $(\text{coh}(X))$ can be identified to $X$ itself, 
but $D^b(\text{coh}(X))$ has
more symmetries and the investigation of its structure may reveal something more fundamental.
If $X$ is a smooth projective variety, then $D^b(\text{coh}(X))$ satisfies finiteness properties such as 
finite homological dimension, Home-finiteness, and saturatedness (\cite{BvdB}).
A birational map like blowing up is reflected to a semi-orthogonal decomposition (SOD) (\cite{BO}).
Even singular varieties with quotient singularities can be considered in the same way by replacing them
by smooth Deligne-Mumford stacks (\cite{DK}, \cite{SDG}), and we observe the parallelism between the 
minimal model program (MMP) and SOD.   

But recently we found that the category $D^b(\text{coh}(X))$ for a singular variety $X$ 
may have nice SOD when the singularities are 
not so bad (\cite{NC}, \cite{Kuznetsov}, \cite{KKS}) at least in the case of surfaces.
The paper \cite{KKS} shows that the structure of $D^b(\text{coh}(X))$ 
is quite interesting at least in the case of 
a rational surface with cyclic quotient singularities.

In this paper we calculate the case of dimension $3$.
We expect that there are still richer structures in dimension $3$. 
We mainly consider $3$-fold with an ordinary double point (ODP).
There are two cases; $\mathbf{Q}$-factorial and non-$\mathbf{Q}$-factorial.  
We will see the difference in the following.
We note that an ordinary double point is $\mathbf{Q}$-factorial if and only if it is (locally) factorial.
We also note that an ordinary double point on a rational surface is never factorial, but there are rational 
$3$-folds with factorial and non-factorial ODP's. 

\vskip 1pc

We first recall some definitions and theorems on generators of categories in \S2 
and the triangulated category of singularities $D_{\text{sg}}(X)$ of a variety $X$ in \S3.
We calculate triangulated categories of singularities in the case of ODP in \S4.
Then we prove the following in \S5:

\begin{Thm}[= Theorem~\ref{main}]
Let $X$ be a Gorenstein projective variety, let $L$ be a maximally Cohen-Macaulay coherent sheaf on $X$ 
which generates $D_{\text{sg}}(X)$, and let $F$ be a coherent sheaf which is constructed from $L$ 
by a flat non-commutative (NC) deformation over the 
endomorphism algebra $R = \text{End}(F)$ such that $\text{Hom}(F,F[p]) = 0$ for all $p \ne 0$.
Then $D^b(\text{coh}(X))$ has an SOD into the triangulated subcategory generated by $F$,  
which is equivalent to $D^b(\text{mod-}R)$, and 
its right orthogonal complement which is saturated, i.e., 
the complement is similar to the case of a smooth projective variety.
\end{Thm}

We apply the theorem to the case where $X$ is a $3$-fold with only one ODP 
which is not $\mathbf{Q}$-factorial, and 
calculate the structure of $D^b(\text{coh}(X))$ in \S6:

\begin{Thm}[= Theorem~\ref{ODP}]
Let $X$ be a $3$-dimensional projective variety with only one ODP which is not $\mathbf{Q}$-factorial.
Assume that there are reflexive sheaves $L_1,L_2$ of rank $1$ such that 
$\dim H^0(X, \mathcal{H}om(L_i,L_j)) = \delta_{ij}$ and that 
$H^p(X, \mathcal{H}om(L_i,L_j)) = 0$ for $p > 0$.
Assume moreover that $L_1,L_2$ generate the triangulated category of singularities $D_{\text{sg}}(X)$.
Then $L_1,L_2$ generate an admissible subcategory of $D^b(\text{coh}(X))$, 
which is equivalent to $D^b(\text{mod-}R)$ 
for a $4$-dimensional non-commutative algebra $R$, such that its right orthogonal complement is 
a saturated category.  
\end{Thm}

In \S7, we calculate some examples.
We give two examples where the assumptions of Theorem~\ref{ODP} are satisfied, 
and then we consider examples where the singularities are factorial ODP's.
In the latter case, non-commutative (NC) deformations (\cite{NC}) 
of $L$ do not terminate and do not yield suitable coherent sheaf $F$.
Indeed the versal deformation becomes a quasi-coherent sheaf corresponding to an infinite chain of 
coherent sheaves. 
In the appendix, we make a correction to an error in a cited paper \cite{NC} on NC deformations.

We assume that the base field $k$ is algebraically closed of characteristic $0$ in this paper.

\vskip 1pc

I would like to dedicate this paper to Professor Miles Reid for the long lasting friendship 
since he was a Pos Doc and 
I was a graduate student in Tokyo (he kindly corrected my English in my master's thesis at that time).

The author would like to thank Professor Keiji Oguiso for a help in Example~\ref{K3}, and 
Professor Alexander Kuznetsov for the correction in Appendix.

This work was partly done while the author stayed at Korea Advanced Institute 
of Science and Technology (KAIST) 
and National Taiwan University.
The author would like to thank Professor Yongnam Lee, Professor Jungkai Chen, Department of 
Mathematical Sciences of KAIST and National Center for Theoretical Sciences of Taiwan 
for the hospitality and 
excellent working conditions.

This work was partly supported by Grant-in-Aid for Scientific Research (A) 16H02141.
This was also partly supported by the National Science Foundation Grant DMS-1440140 
while the author stayed at the Mathematical Sciences Research Institute in Berkeley 
during the Spring 2019 semester.

\section{Generators}

We collect some definitions and results concerning generators of categories 
and representabilities of functors.
$T$ denotes a triangulated category in this section.

A set of objects $E \subset T$ is said to be a {\em generator} 
if the right orthogonal complement defined by 
$E^{\perp} = \{A \in T \mid \text{Hom}_T(E,A[p]) = 0 \,\,\forall p\}$ is zero: $E^{\perp} = 0$.
$E$ is said to be a {\em classical generator} if $T$ coincides with the smallest triangulated subcategory
which contains $E$ and closed under direct summands. 
$E$ is saids to be a {\em strong generator} if there exists a number $n$ such that $T$ coincides with the 
full subcategory obtained from $E$ by taking finite direct sums, direct summands, shifts, 
and at most $n-1$ cones.

$T$ is said to be {\em Karoubian} if every projector splits.
The {\em idempotent completion} or the {\em Karoubian envelope} 
of a triangulated category is defined to be the category 
consisting of all kernels of all projectors.

A triangulated full subcategory $B$ (resp. $C$) of $T$ is said to be {\em right (resp. left) admissible} 
if the natural functor 
$F: B \to T$ (resp. $F': C \to T$) has a right (resp. left) adjoint functor $G: T \to B$ (resp. $G': T \to C$).
An expression 
\[
T = \langle C,B \rangle
\]
is said to be a {\em semi-orthogonal decomposition (SOD)} 
if $B,C$ are triangulated full subcategories such that $\text{Hom}_T(b,c) = 0$ 
for any $b \in B$ and $c \in C$ 
and such that, for any $a \in T$, there exists a distinguished 
triangle $b \to a \to c \to b[1]$ for some $b \in B$ and $c \in C$. 
In this case, $B$ (resp. $C$) is a right (resp. left) admissible subcategory.
Conversely, if $B$ (resp. $C$) is a right (resp. left) admissible subcategory, then there is a semi-orthogonal
decomposition $T = \langle C,B \rangle$ for $C = B^{\perp}$ (resp. $B = {}^{\perp}C$)  (\cite{Bondal}).

$T$ is said to be {\em Hom-finite} if $\sum_i \dim \text{Hom}(A,B[i]) < \infty$ for any objects $A,B \in T$.
A cohomological functor $H: T^{\text{op}} \to (\text{Mod-}k)$ is said to be of {\em finite type}
if $\sum_i  \dim H(A[i]) < \infty$ for any object $A \in T$.
A Hom-finite triangulated category $T$ is said to be {\em right saturated} 
if any cohomological functor $H$ 
of finite type is representable by some object $B \in T$.
A right saturated full subcategory of a Hom-finite triangulated category is automatically right admissible 
and yields a semi-orthogonal decomposition (\cite{Bondal}).
In a similar way, we define homological functor of finite type 
and a left saturated category which is automatically left admissible.
If $T$ is Hom-finite, has a strong generator, and Karoubian, then $T$ is right saturated 
(\cite{BvdB} Theorem 1.3).

Let $T$ be a triangulated category which has arbitrary coproducts (e.g., infinite direct sums).
An object $A \in T$ is said to be {\em compact} if 
\[
\coprod_{\lambda} \text{Hom}(A,B_{\lambda}) \cong \text{Hom}(A, \coprod_{\lambda} B_{\lambda})
\]
for all coproducts $\coprod_{\lambda} B_{\lambda}$.
We denote by $T^c$ the full subcategory of $T$ consisting of all compact objects.
If $X$ is a quasi-separated and quasi-compact scheme and $T = D(\text{Qcoh}(X))$, 
then $T^c = \text{Perf}(X)$, the triangulated subcategory of 
perfect complexes (\cite{BvdB} Theorem 3.1.1).

$T$ is said to be {\em compactly generated} if $(T^c)^{\perp} = 0$.
If $T$ is compactly generated, then $E \in T^c$ classically generates $T^c$ 
if and only if $E$ generates $T$
(\cite{RN}, \cite{BvdB} Theorem 2.1.2).

\begin{Thm}[\cite{Neeman} Theorem 4.1 (Brown representability theorem)]
Let $T$ be a compactly generated triangulated category and $T'$ be another triangulated category.
Let $F: T \to T'$ be an exact functor which commutes with coproduct:
\[
\coprod_{\lambda} F(A_{\lambda}) \cong F(\coprod_{\lambda} A_{\lambda}).
\]
Then there exists a right adjoint functor $G: T' \to T$:
\[
\text{Hom}_{T'}(A,G(B)) \cong \text{Hom}_T(F(A),B).
\]
\end{Thm}

\section{Orlov's triangulated category of singularities}

We recall Orlov's theory of the triangulated category of singularities. 
Let $X$ be a separated noetherian scheme of finite Krull dimension whose category of coherent sheaves 
has enough locally free sheaves, e.g., a quasi-projective variety.
Orlov defined a {\em triangulated category of singularities} $D_{\text{sg}}(X)$ 
as the quotient category of the bounded 
derived category of coherent sheaves $D^b(\text{coh}(X))$ 
by the category of perfect complexes $\text{Perf}(X)$:
$D_{\text{sg}}(X) = D^b(\text{coh}(X))/\text{Perf}(X)$.

The triangulated category of singularities behaves well when $X$ is Gorenstein:

\begin{Prop}[\cite{Orlov1} Proposition 1.23]\label{Orlov-Prop1}
Assume that $X$ is Gorenstein.
Then any object in $D_{\text{sg}}(X)$ is isomorphic to the image of a coherent sheaf $F$ such that 
$\mathcal{H}om(F,\mathcal{O}_X[i]) = 0$ for all $i > 0$.
\end{Prop}

If $X$ is Gorenstein, then the natural morphism 
$F \to R\mathcal{H}om(R\mathcal{H}om(F,\mathcal{O}_X),\mathcal{O}_X)$ 
is an isomorphism.
If $(R,M)$ is a Gorenstein complete local ring of dimension $d$, then the local duality theorem says that 
\[
\text{Ext}^i_R(F,R) \cong \text{Hom}_R(H_M^{d-i}(F),E) 
\]
for any $R$-module $F$, where $E$ is an injective hull of $k = R/M$.
Thus the condition $\mathcal{H}om(F,\mathcal{O}_X[i]) = 0$ for all $i > 0$ is equivalent to saying that 
$F$ is a maximally Cohen-Macaulay (MCM) sheaf.

\begin{Thm}[\cite{Orlov1} Proposition 1.21]\label{Orlov-Prop2}
Assume that $X$ is Gorenstein.
Let $F$ be coherent sheaf such that $\mathcal{H}om(F,\mathcal{O}_X[i]) = 0$ for all $i > 0$.
Let $G \in D^b(\text{coh}(X))$, and let
$N$ be an integer such that $\text{Hom}(P,G[i]) = 0$ for all locally free sheaves $P$ and all $i > N$, 
e.g., $N = 0$ if $G$ is a sheaf and $X$ is affine.
Then 
\[
\text{Hom}_{D_{\text{sg}}}(F,G[N]) \cong \text{Hom}_{D^b(\text{coh}(X))}(F,G[N])/R
\]
where $R$ is the subspace of morphisms which factor through locally free sheaves $P$ such as 
$F \to P \to G[N]$.
\end{Thm}

\begin{proof}
\cite{Orlov1} Proposition 1.21 assumes that $G$ is a sheaf.
But this assumption is not used in the proof.
We note that $N$ depends on $G$.
\end{proof}

\begin{Thm}[Kn\"orrer periodicity \cite{Orlov1} Theorem 2.1]\label{Knorrer}
Let $V$ be a separated regular noetherian scheme of finite Krull dimension 
(e.g., a smooth quasi-projective variety)
and let $f: V \to \mathbf{A}^1$ be a flat morphism.
Let $W = V \times \mathbf{A}^2$ and let $g = f + xy: W \to \mathbf{A}^1$ 
for coordinates $(x, y)$ on $\mathbf{A}^2$.
Let $X$ (resp. $Y$) be the fiber of $f$ (resp. $g$) above $0$.
Let $Z = \{f = x = 0\} \subset Y$, and let $i: Z \to Y$ and $q: Z \to X$ be natural morphisms.
Then $\Phi = Ri_*q^*: D^b(\text{coh}(X)) \to D^b(\text{coh}(Y))$ induces an equivalence
$\bar \Phi: D_{\text{sg}}(X) \to D_{\text{sg}}(Y)$.
\end{Thm}

\begin{Thm}[\cite{Orlov2} Theorem 2.10]
Let $X$ and $X'$ be quasi-projective varieties.
Assume that the formal completions $\hat X$ and $\hat X'$ along their singular loci are isomorphic.
Then the idempotent completions of the triangulated categories of singularities 
$\overline{D_{\text{sg}}(X)}$ and $\overline{D_{\text{sg}}(X')}$ are equivalent.
\end{Thm}

\section{Triangulated categories of singularities for ordinary double points}

We calculate triangulated categories of singularities for varieties with only ordinary double points.

\begin{Expl}[\cite{Orlov1} Subsection 3.3]
Let $X_n = \{x_0^2+\dots+x_n^2 = 0\} \subset \mathbf{A}^{n+1}$ 
be an ordinary double point of dimension $n$.
Then $D_{\text{sg}}(X_n) \cong D_{\text{sg}}(X_{n+2})$ by the Kn\"orrer periodicity 
(\cite{Orlov1} Theorem 2.1).

We consider $X_0$.
Let $A = k[z]/z^2$.
Then any object of $D_{\text{sg}}(X_0)$ is represented by a finite $A$-module.
Therefore it is a direct sum of $V_1 = k = A/(z)$.

A natural exact sequence $0 \to (z) \to A \to A/(z) \to 0$ yields an exact triangle
$V_1 \to A \to V_1 \to V_1[1]$ in $D^b(\text{coh}(X_0))$, hence an isomorphism
$V_1 \cong V_1[1]$ in $D_{\text{sg}}(X_0)$.

Therefore we have $\text{Hom}_{D_{\text{sg}}(X_0)}(V_1,V_1) \cong k$ 
which is generated by the identity $\text{Id}$.
It follows that $D_{\text{sg}}(X_0)$ is already idempotent complete.

The translation functor takes $V_1 \mapsto V_1$ and $\text{Id} \mapsto \text{Id}$.
The exact triangle are only trivial ones.
\end{Expl} 

\begin{Expl}
We consider $X_1$.

Let $B = k[z,w]/(zw)$.
Then any object of $D_{\text{sg}}(X_1)$ is represented by a finite $B$-module $M$ such that 
$\text{Hom}(M,B[i]) = 0$ for all $i > 0$.
Therefore it is a direct sum of $M_z = B/(w)$ and $M_w = B/(z)$.

A natural exact sequence $0 \to zB \to B \to B/(z) \to 0$ yields an exact triangle
$M_z \to B \to M_w \to M_z[1]$ in $D^b(\text{coh}(X_1))$, hence an isomorphism
$M_w \cong M_z[1]$ in $D_{\text{sg}}(X_1)$.
We also have $M_z \cong M_w[1]$ in $D_{\text{sg}}(X_1)$.

We have $\text{Hom}_{D^b(\text{coh}(X_1))}(M_z,M_z) \cong k[z]$.
Since the multiplication map $z: M_z \to M_z$ is factored as $M_z \cong zB \subset B \to B/(w)$, 
we have $z \in R$ in the notation of 
Theorem~\ref{Orlov-Prop2}.
Therefore we have $\text{Hom}_{D_{\text{sg}}(X_1)}(M_z,M_z) \cong k$ 
which is generated by the identity $\text{Id}_z$.

Since $\text{Hom}_{D^b(\text{coh}(X_1))}(M_z,M_w) = 0$, 
we have $\text{Hom}_{D_{\text{sg}}(X_1)}(M_z,M_w) = 0$. 
It follows that $D_{\text{sg}}(X_1)$ is already idempotent complete.

The translation functor takes $M_z \mapsto M_w$ and $\text{Id}_z \mapsto \text{Id}_w$.
The exact triangle are only trivial ones.
\end{Expl} 

\begin{Expl}
We consider another $1$-dimensional scheme $Y_1$ whose singularity is analytically 
isomorphic to that of $X_1$ but not algebraically.

Let $C = k[z,w]/(z^2+z^3 +w^2)$.
The completion of $Y_1 = \text{Spec}(C)$ at the singularity is isomorphic to that of $X_1$, hence  
$\overline{D_{\text{sg}}(Y_1)} \cong \overline{D_{\text{sg}}(X_1)}$.
But we will see that $D_{\text{sg}}(Y_1)$ is not equivalent to $D_{\text{sg}}(X_1)$
(\cite{Orlov2} Introduction).

Let $C' \to C$ be the normalization.
Then $C' \cong k[t]$ with $z = -t^2-1$ and $w = -t^3-t$.

Any object of $D_{\text{sg}}(Y_1)$ is represented by a finite $C$-module $N$ such that 
$\text{Hom}(N,C[i]) = 0$ for all $i > 0$.
Therefore it is a direct sum of $C'$.

There are $2$ points of $Y'_1=\text{Spec}(C') \cong \mathbf{A}^1$ above the singular point of $Y$; 
we have $(z,w)C' = P \cap Q$ for prime ideals $P = (t+\sqrt{-1})$ and $Q = (t-\sqrt{-1})$ of $C'$.
There is a surjective homomorphism $C^{\oplus 2} \to C'$ given by $(a,b) \mapsto a-bt$.
The kernel of this homomorphism is equal to $(w,z)C'$, which is isomorphic to $C'$ as a $C$-module.
Indeed we have $w-zt=0$ and $wt=-z^2-z$, etc.
Therefore we have an exact sequence 
\[
0 \to (w,z)C' \to C^{\oplus 2} \to C' \to 0
\]
yielding an exact triangle
$C' \to C^{\oplus 2} \to C' \to C'[1]$ in $D^b(\text{coh}(Y_1))$, hence an isomorphism
$C' \cong C'[1]$ in $D_{\text{sg}}(Y_1)$.

We have $\text{Hom}_{D^b(\text{coh}(Y_1))}(C',C') \cong C' \cong k[t]$ as $C$-modules, 
where a $C$-module homomorphism on the left hand side is mapped to the image of $1$ by the 
homomorphism.
We note that a $C$-homomorphism is determined by the image of $1$ because $C' \to C$ is birational.
Since the multiplication map $z: C' \to C'$ is factored as 
$C' \cong zC' \subset C \to C'$, we have $z \in R$ in the notation of 
Theorem~\ref{Orlov-Prop2}.
Therefore we have $\text{Hom}_{D_{\text{sg}}(Y_1)}(C',C') \cong k[t]/(t^2+1)$.

The translation functor takes $C' \mapsto C'$.
The exact triangle are only trivial ones.

There is an idempotent $(1 \pm \sqrt{-1}t)/2 \in \text{Hom}_{D_{\text{sg}}(Y_1)}(C',C')$.
But there is no corresponding sheaf on $Y_1$.
Therefore $D_{\text{sg}}(Y_1)$ is not idempotent complete.
The corresponding idempotent completion is equivalent to $D_{\text{sg}}(X_1)$. 
\end{Expl}

\begin{Expl}
We consider $X_2$.

We rewrite the equation of $X_2$ to $xy + z^2 = 0$.
There is an equivalence $\Phi_1: D_{\text{sg}}(X_0) \cong D_{\text{sg}}(X_2)$ (Theorem~\ref{Knorrer}) 
which is given as follows.
Let $Z_1 = \text{Spec}(k[y,z]/z^2)$.
Then there are natural morphisms $q_1: Z_1 \to X_0$ and $i_1: Z_1 \subset X_2$.
The equivalence is given by $\Phi_1 = Ri_{1*}q_1^*$.

Let $L = \{x = z = 0\}$ be a line on the surface $X_2$ through the origin.
$L$ is a prime divisor which is not a Cartier divisor, but $2L$ is a Cartier divisor.

Since $V_1 = A/(z)$, we have $\Phi_1(V_1) = Ri_{1*}(k[y,z]/(z)) = k[x,y,z]/(x,z) = \mathcal{O}_L$.
We have an exact sequence
\[
0 \to \mathcal{O}_{X_2}(-L) \to \mathcal{O}_{X_2} \to \mathcal{O}_L \to 0.
\]
Thus $\mathcal{O}_L \cong \mathcal{O}_{X_2}(-L)[1]$ in $D_{\text{sg}}(X_2)$.
Therefore $D_{\text{sg}}(X_2)$ is spanned by a reflexive sheaf $\mathcal{O}_{X_2}(-L)$ of rank $1$.

We note that $\mathcal{O}_L$ is a torsion sheaf, but $\mathcal{O}_{X_2}(-L)$ is a Cohen-Macaulay sheaf.

There is an exact sequence
\[
0 \to \mathcal{O}_{X_2}(-L) \to F \to \mathcal{O}_{X_2}(-L) \to 0
\]
for a locally free sheaf $F$ of rank $2$ (\cite{NC} Example 5.5).
Thus $\mathcal{O}_{X_2}(-L) \cong \mathcal{O}_{X_2}(-L)[1]$ in $D_{\text{sg}}(X_2)$.
\end{Expl}

\begin{Expl}
We consider $X_3$.
This is the case of a non-$\mathbf{Q}$-factorial $3$-fold.

We rewrite the equation of $X_3$ to $xy + zw = 0$.
There is an equivalence $\Phi_2: D_{\text{sg}}(X_1) \cong D_{\text{sg}}(X_3)$ (Theorem~\ref{Knorrer}) 
which is given as follows.
Let $Z_2 = \text{Spec}(k[y,z,w]/zw)$.
Then there are natural morphisms $q_2: Z_2 \to X_1$ and $i_2: Z _2 \subset X_3$.
The equivalence is given by $\Phi_2 = Ri_{2*}q_2^*$.

Let $L = \{x=w=0\}$ and $L' = \{x=z=0\}$ be planes on $X_3$ through the origin.
They are prime divisors which are not $\mathbf{Q}$-Cartier divisors.

Since $M_z = B/(w)$ and $M_w = B/(z)$, 
we have $\Phi_2(M_z) = Ri_{2*}(k[y,z,w]/(w)) = k[x,y,z,w]/(x,w) = \mathcal{O}_L$
and $\Phi_2(M_w) = Ri_{2*}(k[y,z,w]/(z)) = k[x,y,z,w]/(x,z) = \mathcal{O}_{L'}$
We have an exact sequence
\[
0 \to \mathcal{O}_{X_3}(-L) \to \mathcal{O}_{X_3} \to \mathcal{O}_L \to 0.
\]
Thus $\mathcal{O}_L \cong \mathcal{O}_{X_3}(-L)[1]$ in $D_{\text{sg}}(X_3)$.
We also have $\mathcal{O}_{L'} \cong \mathcal{O}_{X_3}(-L')[1]$ in $D_{\text{sg}}(X_3)$.
Therefore $D_{\text{sg}}(X_3)$ is spanned by reflexive sheaves $\mathcal{O}_{X_3}(-L)$ and 
$\mathcal{O}_{X_3}(-L')$ of rank $1$.

We note that $\mathcal{O}_L$ and $\mathcal{O}_{L'}$ are torsion sheaves, but $\mathcal{O}_{X_3}(-L)$ 
and $\mathcal{O}_{X_3}(-L')$ are 
Cohen-Macaulay sheaves.

There is an exact sequence
\[
0 \to \mathcal{O}_{X_3}(-L) \to F \to \mathcal{O}_{X_3}(-L') \to 0
\]
for a locally free sheaf $F$ of rank $2$ (\cite{NC} Example 5.6).
Thus $\mathcal{O}_{X_3}(-L') \cong \mathcal{O}_{X_3}(-L)[1]$ in $D_{\text{sg}}(X_3)$.
We also have $\mathcal{O}_{X_3}(-L) \cong \mathcal{O}_{X_3}(-L')[1]$ in $D_{\text{sg}}(X_3)$.
\end{Expl}

\begin{Expl}
We consider a $3$-dimensional scheme $Y_3$ which is analytically isomorphic to $X_3$ 
at the singular points but not algebraically.
$Y_3$ is $\mathbf{Q}$-factorial, hence locally factorial, because the fundamental group 
of the punctured neighborhood of the singularity is trivial.

Let $Y_3$ be a $3$-fold defined by an equation $xy+z^2+z^3+w^2 = 0$.
$Y_3$ has an ordinary double point which is $\mathbf{Q}$-factorial. 

There is an equivalence $\Phi'_2: D_{\text{sg}}(Y_1) \cong D_{\text{sg}}(Y_3)$ 
(Theorem~\ref{Knorrer}) which is given as follows.
Let $Z'_2 = \text{Spec}(k[y,z,w]/(z^2+z^3+w^2))$.
Then there are natural morphisms $q'_2: Z'_2 \to Y_1$ and $i'_2: Z' _2 \subset Y_3$.
The equivalence is given by $\Phi'_2 = Ri'_{2*}(q'_2)^*$.

Let $D = \{x=z^2+z^3+w^2=0\}$ be a prime divisor on $Y_3$, which is a Cartier divisor.
Let $D' \to D$ be the normalization.
On $Y_1$, we have $C' = k[t]$ with $z = -t^2-1$ and $w = -t^3-t$.
On $Y_3$, we have $\mathcal{O}_{D'} = k[y,t]$.
Thus we have $\Phi'_2(C') = Ri_{2*}(k[y,t]) = k[y,t] = \mathcal{O}_{D'}$.

There are surjective homomorphisms 
$\mathcal{O}_{Y_3}^{\oplus 2} \to \mathcal{O}_D^{\oplus 2} \to \mathcal{O}_{D'}$, 
hence an exact sequence 
\[
0 \to F \to \mathcal{O}_{Y_3}^{\oplus 2} \to \mathcal{O}_{D'} \to 0
\]
defining a coherent sheaf $F$.
Then $D_{\text{sg}}(Y_3)$ is spanned by $F$. 
The completion of $Y_3$ at the singular point is isomorphic to that of $X_3$,
and the completion of $F$ corresponds to that of 
$\mathcal{O}_{X_3}(-L) \oplus \mathcal{O}_{X_3}(-L')$.
\end{Expl}

\section{Semi-orthogonal decomposition for a Gorenstein projective variety}

The following is the main result of this section:

\begin{Thm}\label{main}
Let $X$ be a Gorenstein projective variety, let $L$ be a maximally Cohen-Macaulay sheaf on $X$, 
let $F$ be a coherent sheaf which is a perfect complex on $X$,
and let $R = \text{End}_X(F)$ be the endomorphism ring. 
Assume the following conditions:

\begin{enumerate}
\item $\text{Hom}_{D^b(\text{coh}(X))}(F,F[p]) = 0$ for $p \ne 0$.

\item $F$ is flat over $R$.

\item $L$ generates the triangulated category of singularities 
$D_{\text{sg}}(X) = D^b(\text{coh}(X))/\text{Perf}(X)$ in the sense 
that $\text{Hom}(L, A) \ne 0$ for any $A \not\cong 0$ in $D_{\text{sg}}(X)$ 
(note that there is no shift of $A$).

\item $L$ belongs to the triangulated subcategory $T$ of $D^b(\text{coh}(X))$ generated by $F$.
\end{enumerate}

Denote by $T^{\perp}$ the right orthogonal complement of $T$ in 
$D^b(\text{coh}(X))$, the full subcategory
consisting of objects $A$ such that $\text{Hom}_{D^b(\text{coh}(X))}(F,A[p]) \cong 0$ for all $p$.
Then the following hold.

\begin{enumerate}
\item There is an equivalence $T \cong D^b(\text{mod-}R)$, the bounded 
derived category of finitely generated right $R$-modules.

\item There is a semi-orthogonal decomposition
\[
D^b(\text{coh}(X)) = \langle T^{\perp},T \rangle.
\]

\item $T^{\perp} \subset \text{Perf}(X)$.

\item $T^{\perp}$ is two-sided saturated.
\end{enumerate}
\end{Thm}

\begin{proof}
(1) and (2). 
We define functors between unbounded triangulated categories 
$\Phi: D(\text{Mod-}R) \to D(\text{Qcoh}(X))$ and $\Psi: D(\text{Qcoh}(X)) \to D(\text{Mod-}R)$
in the following, where $D(\text{Mod-}R)$ denotes the unbounded derived category of 
right $R$-modules which are not necessarily finitely generated.
We set $\Phi(\bullet) = \bullet \otimes_R^L F$, 
where lower $R$ stands for the tensor product over $R$ and 
the upper $L$ for the 
left derived functor, and $\Psi(\bullet) = R\text{Hom}_X(F,\bullet)$. 
That is, we define $\Phi(A) = P_* \otimes_R F$ as complexes for a $K$-projective resolution 
$P_* \to A$ in $D(\text{Mod-}R)$, 
and $\Psi(B) = \text{Hom}_X(F,I_*)$ as complexes for a $K$-injective resolution 
$B \to I_*$ in $D(\text{Qcoh}(X))$.
Since $F$ is flat over $R$ and perfect on $X$, these functors induces
functors $\Phi_0: D^b(\text{mod-}R) \to D^b(\text{coh}(X))$ and 
$\Psi_0: D^b(\text{coh}(X)) \to D^b(\text{mod-}R)$, 
where $D^b(\text{mod-}R)$ denotes the bounded derived category of 
right $R$-modules which are finitely generated.

We have
\[
\begin{split}
&\text{Hom}_X(\Phi(A),B) \cong \text{Hom}_X(P_* \otimes_R F,I_*) \\
&\cong \text{Hom}_R(P_*, \text{Hom}_X(F,I_*)) \cong \text{Hom}_R(A, \Psi(B)).
\end{split}
\]
Therefore $\Phi$ and $\Psi$ are adjoints.

We have 
\[
\Psi\Phi(A) = R\text{Hom}_X(F,P_* \otimes_R F) \cong P_* \otimes_R R\text{Hom}(F,F) 
\cong P_* \cong A.
\]
Thus adjunction morphism $A \to \Psi\Phi(A)$ is an isomorphism, and $\Phi$ is fully faithful.
Let $T'$ be the image of $\Phi$, i.e., the triangulated subcategory of $D(\text{Qcoh}(X))$ 
generated by $F$.
Then we conclude that there is a semi-orthogonal decomposition 
$D(\text{Qcoh}(X)) = \langle (T')^{\perp},T' \rangle$.

By restriction, we deuce that $\Phi_0$ is fully faithful and $\Psi_0$ is its right adjoint.
Therefore we have (1) and (2).
We needed the unbounded version of the statement for the following proof of the assertion (4).

\vskip 1pc

(3) Let $G \in D^b(\text{coh}(X))$.
If $G \not\in \text{Perf}(X)$, then 
$\bar G[N] \not\cong 0 \in D_{\text{sg}}(X)$, where $N$ is the number in Theorem~\ref{Orlov-Prop2}
and $\bar G$ denotes the object $G$ in $D_{\text{sg}}(X)$.
Then $\text{Hom}_{D_{\text{sg}}(X)}(\bar L,\bar G[N]) \ne 0$, 
hence $\text{Hom}_{D^b(\text{coh}(X))}(L,G[N]) \ne 0$ by Theorem~\ref{Orlov-Prop2}.
Thus $G \not\in T^{\perp}$.
Therefore $T^{\perp} \subset \text{Perf}(X)$.

\vskip 1pc

(4) We modify the proof of \cite{BvdB} Theorem A.1.
We note that $T^{\perp}$ is Hom-finite because it is contained in $\text{Perf}(X)$.
It has also a Serre functor given by $\otimes \omega_X[\dim X]$.
Hence it is sufficient to prove that $T^{\perp}$ is right saturated.

$(T')^{\perp}$ has arbitrary coproduct, and it is compactly generated by 
$T^{\perp} = (T')^{\perp} \cap \text{Perf}(X)$.
Indeed, for any object $0 \not\cong B \in (T')^{\perp}$, 
there is $A \in \text{Perf}(X)$ such that $\text{Hom}_X(A,B) \ne 0$.
By (2), there is a left adjoint functor $\Xi: D(\text{Qcoh}(X)) \to (T')^{\perp}$ of the inclusion functor 
$\Theta: (T')^{\perp} \to D(\text{Qcoh}(X))$ which
induces a left adjoint functor $\Xi_0: D^b(\text{coh}(X)) \to T^{\perp}$ of the inclusion functor 
$\Theta_0: T^{\perp} \to D^b(\text{coh}(X))$.
Then $\text{Hom}_X(\Xi_0(A),B) \ne 0$ because $\text{Hom}_X(F,B) = 0$.

We use \cite{CKN} Lemma 2.14.
Let $H: (T^{\perp})^{\text{op}} \to (\text{mod-}k)$ be any cohomological functor of finite type.
We define $G = DH: T^{\perp} \to (\text{mod-}k)$ using the duality functor 
$D: (\text{mod-}k)^{\text{op}} \to (\text{mod-}k)$ given by 
$D(V) = \text{Hom}(V,k)$.
Let $G': (T')^{\perp} \to (\text{mod-}k)$ be the Kan extension of $G$ given by
$G'(B) = \text{colim}_{C \to B}G(C)$, where the colimit is taken for all morphisms 
from all compact objects $C \in T^{\perp}$.
$DG'$ is represented by an object $A \in (T')^{\perp}$ by the Brown representability theorem.
Since $DDH = H$ on $T^{\perp}$, we deduce that $H$ is represented by $A$.

We have to prove that $A \in T^{\perp} = (T')^{\perp} \cap D^b(\text{coh}(X))$.
We take an embedding $p: X \to \mathbf{P}^n$, 
and let $H'=H\Xi_0p^*: D^b(\text{coh}(\mathbf{P}^n))^{\text{op}} \to (\text{mod-}k)$ 
be the induced cohomological functor. 
By Beilinson's theorem (\cite{Beilinson}), 
there is an equivalence $t': D(\text{Mod-}S) \cong D(\text{Qcoh}(\mathbf{P}^n))$ 
which induces an equivalence 
$t: D^b(\text{mod-}S) \cong D^b(\text{coh}(\mathbf{P}^n))$
for some finite dimensional associative ring $S$.
Let $H'' = H't: D^b(\text{mod-}S)^{\text{op}} \to (\text{mod-}k)$ and 
$A' = (t')^{-1}p_*\Theta(A) \in D(\text{Mod-}S)$.
Then we have $H''(B) = \text{Hom}_X(\Xi_0p^*t(B),A) = \text{Hom}(B, A')$ for any 
$B \in D^b(\text{mod-}S)$, i.e., 
$H''$ is represented by $A'$.
Therefore our assertion is reduced to showing that $A' \in D^b(\text{mod-}S)$.
We have $\sum \dim \text{Hom}_S(S[n],A') = \sum H''(S[n]) < \infty$, 
hence $A' \in D^b(\text{mod-}S)$, and we are done.
\end{proof}

\begin{Rem}
We will use the theorem in the case where $F$ is obtained as a versal non-commutative deformation of 
a simple collection $L$ as described in \cite{NC}.
In this case $F$ is flat over the parameter algebra $R = \text{End}_X(F)$.

We do not know how to generalize the theorem in the case where $F$ is not necessarily flat over its
endomorphism ring.
Indeed we do not know how to prove that the functor
$\Phi: D^-(\text{mod-}R) \to D^-(\text{coh}(X))$ defined by 
$\Phi(\bullet) = \bullet \otimes_R^L F$ is bounded, i.e., 
$\Phi$ sends $D^b(\text{mod-}R)$ to $D^b(\text{coh}(X))$, though its right adjoint functor 
$\Psi: D^-(\text{coh}(X)) \to D^-(\text{mod-}R)$ defined by 
$\Psi(\bullet) = R\text{Hom}(F,\bullet)$ is bounded.
\end{Rem}

\section{Non-commutative deformation on a $3$-fold 
with a non-$\mathbf{Q}$-factorial ordinary double point}

We will apply Theorem~\ref{main} to a $3$-fold 
with a non-$\mathbf{Q}$-factorial ordinary double point.
The following theorem says that the conditions of Theorem~\ref{main} are 
satisfied under some assumptions:

\begin{Thm}\label{ODP}
Let $X$ be a projective variety of dimension $3$.
Assume that there is only one singular point $P$ 
which is a non-$\mathbf{Q}$-factorial ordinary double point.
Then there is a $\mathbf{Q}$-factorialization $f: Y \to X$, 
a projective birational morphism from a smooth projective variety 
whose exceptional locus $l$ is a smooth rational curve.
Assume that there are divisors $D_1,D_2$ on $Y$ such that, 
for $L_i = f_*\mathcal{O}_Y(-D_i)$, the following conditions
are satisfied:

\begin{enumerate}

\item $(D_1,l) = 1$ and $(D_2,l) = -1$.

\item $(L_1,L_2)$ is a simple collection, i.e., 
$\dim \text{Hom}(L_i,L_j) = \delta_{ij}$.

\item $H^p(X,f_*\mathcal{O}_Y(-D_i+D_j)) = 0$ for all $p > 0$ and all $i,j$.

\end{enumerate}

\noindent
Then there are locally free sheaves $F_1,F_2$ of rank $2$ on $X$ given by non-trivial extensions
\begin{equation}\label{F_i}
\begin{split}
0 \to L_2 \to F_1 \to L_1 \to 0 \\
0 \to L_1 \to F_2 \to L_2 \to 0 
\end{split}
\end{equation}
such that, for $L = L_1 \oplus L_2$ and $F = F_1 \oplus F_2$, the following assertions hold:

\begin{enumerate}
\item $\text{Ext}^p(F,F) = 0$ for $p > 0$.

\item $F$ is flat over $R := \text{End}(F) \cong \left( \begin{matrix} k & kt \\ kt & k \end{matrix} \right)
\mod t^2$.

\item $L$ is a Cohen-Macaulay sheaf which generates the triangulated category of singularities
$D_{\text{sg}}(X)$.

\item $L$ belong to the triangulated subcategory of $D^b(\text{coh}(X))$ generated by $F$.
\end{enumerate}
\end{Thm}

\begin{proof}
We consider $2$-pointed non-commutative (NC) deformations of a simple collection $(L_1,L_2)$ 
(\cite{NC}).
There is a spectral sequence
\[
H^p(X,\mathcal{E}xt^q(L_i,L_j)) \Rightarrow \text{Ext}^{p+q}(L_i,L_j)
\]
for any $i,j$.
By the condition (3), we deduce that 
\[
\text{Ext}^p(L_i,L_j) \cong H^0(X, \mathcal{E}xt^p(L_i,L_j))
\]
for all $p$.

A neighborhood of the singular point $P \in X$ is analytically isomorphic to 
that of the vertex of the cone over $\mathbf{P}^1 \times \mathbf{P}^1$ 
in $\mathbf{P}^4$ considered in \cite{NC}~Example 5.6 
(see also the last section).
Since the sheaves $L_1$ and $L_2$ here correspond to the sheaves $\mathcal{O}_X(0,-1)$ 
and $\mathcal{O}_X(-1,0)$ there,  
the extension space $\mathcal{E}xt^1(L_1,L_2)$ is isomorphic to 
$\mathcal{E}xt^1(\mathcal{O}_X(0,-1),\mathcal{O}_X(-1,0))$ there.
By the argument there, there exists a locally free sheaf $F_1$ in an analytic neighborhood of 
$P$ with the extension sequence as stated in the theorem, and similarly $F_2$.
Since $\text{Ext}^1(L_i,L_j) \cong H^0(X, \mathcal{E}xt^1(L_i,L_j))$, 
we deduce that these extensions exist globally on $X$. 

We need to calculate $\mathcal{E}xt^p(L_i,L_j)$ for all $p$ at $P$.
For this purpose we need the following calculation:

\begin{Lem}
Let $X'$ be the cone over $\mathbf{P}^1 \times \mathbf{P}^1$ in $\mathbf{P}^4$ 
as in \cite{NC}~Example 5.6
(we use the notation $X'$ in order to avoid a confusion).
Let $G_1,G_2$ be locally free sheaves of rank $2$ on $X'$ defined by non-trivial 
extensions:
\begin{equation}
\begin{split}
0 \to \mathcal{O}_{X'}(-1,0) \to G_1 \to \mathcal{O}_{X'}(0,-1) \to 0 \\
0 \to \mathcal{O}_{X'}(0,-1) \to G_2 \to \mathcal{O}_{X'}(-1,0) \to 0. 
\end{split}
\end{equation}
Then the following hold: 
\[
\begin{split}
&(1) \,\, H^p(X',\mathcal{O}_{X'}(-1,0)) = H^p(X',\mathcal{O}_{X'}(-2,0)) = H^p(X',\mathcal{O}_{X'}(-1,1)) = 0,
\,\, \forall p. \\
&(2) \,\, H^p(\mathcal{O}_{X'}(0,-1)) = H^p(X',\mathcal{O}_{X'}(0,-2)) = H^p(X',\mathcal{O}_{X'}(1,-1)) = 0, 
\,\, \forall p. \\
&(3) \,\, \text{Ext}^p(G_1,\mathcal{O}_{X'}(-1,0)) = \text{Ext}^p(G_1,\mathcal{O}_{X'}(0,-1)) = 0, 
\,\, \forall p > 0. \\
&(4) \,\, \text{Ext}^p(G_2,\mathcal{O}_{X'}(-1,0)) = \text{Ext}^p(G_2,\mathcal{O}_{X'}(0,-1)) = 0, 
\,\, \forall p > 0. \\
&(5) \,\, \dim \text{Hom}(G_1,\mathcal{O}_{X'}(0,-1)) = \dim \text{Hom}(G_2,\mathcal{O}_{X'}(-1,0)) = 1. \\
&(6) \,\, \text{Hom}(G_1,\mathcal{O}_{X'}(-1,0)) = \text{Hom}(G_2,\mathcal{O}_{X'}(0,-1)) = 0. \\
&(7) \,\, \text{Ext}^p(\mathcal{O}_{X'}(0,-1),\mathcal{O}_{X'}(0,-1)) = \text{Ext}^p(\mathcal{O}_{X'}(-1,0),
\mathcal{O}_{X'}(-1,0)) = 0, \,\, 
p > 0, p \equiv 1 (\text{mod }2). \\
&(8) \,\, \text{Ext}^p(\mathcal{O}_{X'}(0,-1),\mathcal{O}_{X'}(-1,0)) = \text{Ext}^p(\mathcal{O}_{X'}(-1,0),
\mathcal{O}_{X'}(0,-1)) = 0, \,\, 
p > 0, p \equiv 0 (\text{mod }2). \\
&(9) \,\, \dim \text{Ext}^p(\mathcal{O}_{X'}(0,-1),\mathcal{O}_{X'}(0,-1)) 
= \dim \text{Ext}^p(\mathcal{O}_{X'}(-1,0),\mathcal{O}_{X'}(-1,0)) = 1, \,\, p > 0, 
p \equiv 0 (\text{mod }2). \\
&(10) \,\, \dim \text{Ext}^p(\mathcal{O}_{X'}(0,-1),\mathcal{O}_{X'}(-1,0)) 
= \dim \text{Ext}^p(\mathcal{O}_{X'}(-1,0),\mathcal{O}_{X'}(0,-1)) = 1, \,\, p > 0, 
p \equiv 1 (\text{mod }2).
\end{split}
\]
\end{Lem}

\begin{proof}
(1) and (2).
Let $D \cong \mathbf{P}^2$ be a prime divisor on $X'$ corresponding to 
$\mathcal{O}_{X'}(1,0)$ such that there is an exact sequence
\[
0 \to \mathcal{O}_{X'}(-1,0) \to \mathcal{O}_{X'} \to \mathcal{O}_D \to 0.
\]
Since $H^p(X',\mathcal{O}_{X'}) \cong H^p(D,\mathcal{O}_D)$ for all $p$, 
we have $H^p(X',\mathcal{O}_{X'}(-1,0)) = 0$ for all $p$.
We have an exact sequence
\[
0 \to \mathcal{O}_{X'}(-2,0) \to \mathcal{O}_{X'}(-1,0) \to \mathcal{O}_D(-P) \to 0
\]
where $\mathcal{O}_D(-P)$ is the ideal sheaf of $P$ on $D$:
\[
0 \to \mathcal{O}_D(-P) \to \mathcal{O}_D \to \mathcal{O}_P \to 0.
\]
Since $H^p(\mathcal{O}_D) \cong H^p(\mathcal{O}_P)$ for all $p$, 
we have $H^p(\mathcal{O}_D(-P)) = 0$ for all $p$.
Then we deduce that $H^p(X',\mathcal{O}_{X'}(-2,0)) = 0$ for all $p$.

Let $S \cong \mathbf{P}^1 \times \mathbf{P}^1$ be a prime divisor on $X'$ 
corresponding to $\mathcal{O}_{X'}(1,1)$ 
such that there is an exact sequence
\[
0 \to \mathcal{O}_{X'}(-1,-1) \to \mathcal{O}_{X'} \to \mathcal{O}_S \to 0.
\]
Then we have
\[
0 \to \mathcal{O}_{X'}(-2,0) \to \mathcal{O}_{X'}(-1,1) \to \mathcal{O}_S(-1,1) \to 0.
\]
Since $H^p(S,\mathcal{O}_S(-1,1)) = 0$ for all $p$, we have $H^p(X',\mathcal{O}_{X'}(-1,1)) = 0$ for all $p$.

The second assertion follows by symmetry.

\vskip 1pc

(3) and (4). 
There are spectral sequences
\[
\begin{split}
&E_2^{p,q} = H^p(X,\mathcal{E}xt^q(\mathcal{O}_{X'}(0,-1),\mathcal{O}_{X'}(-1,0))) 
\Rightarrow \text{Ext}^{p+q}(\mathcal{O}_{X'}(0,-1), \mathcal{O}_{X'}(-1,0)) \\
&E_2^{p,q} = H^p(X,\mathcal{E}xt^q(\mathcal{O}_{X'}(1,0),\mathcal{O}_{X'}(-1,0))) 
\Rightarrow \text{Ext}^{p+q}(\mathcal{O}_{X'}(1,0), \mathcal{O}_{X'}(-1,0)).
\end{split}
\]
Then by (1), we obtain natural isomorphisms
\[
\begin{split}
&\text{Ext}^p(\mathcal{O}_{X'}(0,-1), \mathcal{O}_{X'}(-1,0)) 
\cong H^0(\mathcal{E}xt^p(\mathcal{O}_{X'}(0,-1), \mathcal{O}_{X'}(-1,0)) \\
&\cong \text{Ext}^p(\mathcal{O}_{X'}(1,0), \mathcal{O}_{X'}(-1,0))
\end{split}
\]
for all $p > 0$.

We have a commutative diagram of exact sequences
\[
\begin{CD}
0 @>>> \mathcal{O}_{X'}(-1,0) @>>> G_1 @>>> \mathcal{O}_{X'}(0,-1) @>>> 0 \\
@. @V=VV @VVV @VVV @. \\
0 @>>> \mathcal{O}_{X'}(-1,0) @>>> \mathcal{O}_{X'}^2 @>>> \mathcal{O}_{X'}(1,0) @>>> 0
\end{CD}
\]
where the cokernels of the middle and right vertical arrows are isomorphic to $\mathcal{O}_S(1,0)$.
By (1) we have isomorphisms
\[
\text{Ext}^p(G_1,\mathcal{O}_{X'}(-1,0)) \cong 
\text{Ext}^{p+1}(\mathcal{O}_S(1,0), \mathcal{O}_{X'}(-1,0)) \cong 0
\]
for all $p > 0$.

Since $H^p(X',\mathcal{O}_{X'}) = 0$ for $p > 0$ and $H^p(X',\mathcal{O}_{X'}(-1,-1)) = 0$ 
for all $p$, we also have
\[
\text{Ext}^p(G_1,\mathcal{O}_{X'}(0,-1)) \cong \text{Ext}^{p+1}(\mathcal{O}_S(1,0), 
\mathcal{O}_{X'}(0,-1)) \cong 0
\]
for all $p > 0$.
 
The second assertion follows by symmetry.

\vskip 1pc

(5) etc. 
We have a long exact sequence
\[
\begin{split}
&0 \to \text{Hom}(\mathcal{O}_{X'}(0,-1),\mathcal{O}_{X'}(0,-1)) \to \text{Hom}(G_1,\mathcal{O}_{X'}(0,-1)) 
\to \text{Hom}(\mathcal{O}_{X'}(-1,0),\mathcal{O}_{X'}(0,-1)) \\
&\to \text{Ext}^1(\mathcal{O}_{X'}(0,-1),\mathcal{O}_{X'}(0,-1)) \to \text{Ext}^1(G_1,\mathcal{O}_{X'}(0,-1)) 
\to \text{Ext}^1(\mathcal{O}_{X'}(-1,0),\mathcal{O}_{X'}(0,-1)) \to \dots
\end{split}
\]
Since $\text{Hom}(\mathcal{O}_{X'}(-1,0),\mathcal{O}_{X'}(0,-1)) = 0$ and 
$\text{Ext}^p(G_1,\mathcal{O}_{X'}(0,-1)) = 0$ for $p > 0$,
we deduce 
\begin{equation}\label{coh1}
\begin{split}
&\dim \text{Hom}(G_1,\mathcal{O}_{X'}(0,-1)) = 1, 
\,\, \text{Ext}^1(\mathcal{O}_{X'}(0,-1),\mathcal{O}_{X'}(0,-1)) = 0 \\
&\text{Ext}^p(\mathcal{O}_{X'}(-1,0),\mathcal{O}_{X'}(0,-1)) 
\cong \text{Ext}^{p+1}(\mathcal{O}_{X'}(0,-1),\mathcal{O}_{X'}(0,-1)) \,\, (p > 0).
\end{split}
\end{equation}
We have a long exact sequence
\[
\begin{split}
&0 \to \text{Hom}(\mathcal{O}_{X'}(0,-1),\mathcal{O}_{X'}(-1,0)) 
\to \text{Hom}(G_1,\mathcal{O}_{X'}(-1,0)) \to 
\text{Hom}(\mathcal{O}_{X'}(-1,0),\mathcal{O}_{X'}(-1,0)) \\
&\to \text{Ext}^1(\mathcal{O}_{X'}(0,-1),\mathcal{O}_{X'}(-1,0)) 
\to \text{Ext}^1(G_1,\mathcal{O}_{X'}(-1,0)) \to 
\text{Ext}^1(\mathcal{O}_{X'}(-1,0),\mathcal{O}_{X'}(-1,0))) \to \dots
\end{split}
\]
By construction, the homomorphism 
\[
k \cong \text{Hom}(\mathcal{O}_{X'}(-1,0),\mathcal{O}_{X'}(-1,0)) 
\to \text{Ext}^1(\mathcal{O}_{X'}(0,-1),\mathcal{O}_{X'}(-1,0))
\] 
is injective.
Hence we deduce 
\begin{equation}\label{coh2}
\begin{split}
&\text{Hom}(G_1,\mathcal{O}_{X'}(-1,0)) = 0, 
\,\, \dim \text{Ext}^1(\mathcal{O}_{X'}(0,-1),\mathcal{O}_{X'}(-1,0)) = 1 \\
&\text{Ext}^p(\mathcal{O}_{X'}(-1,0),\mathcal{O}_{X'}(-1,0)) 
\cong \text{Ext}^{p+1}(\mathcal{O}_{X'}(0,-1),\mathcal{O}_{X'}(-1,0))  \,\, (p > 0).
\end{split}
\end{equation}
By symmetry we also obtain
\begin{equation}\label{coh3}
\begin{split}
&\text{Hom}(G_2,\mathcal{O}_{X'}(0,-1)) = 0, \,\, \dim \text{Hom}(G_2,\mathcal{O}_{X'}(-1,0)) = 1 \\ 
&\text{Ext}^1(\mathcal{O}_{X'}(-1,0),\mathcal{O}_{X'}(-1,0)) = 0 
\,\, dim \text{Ext}^1(\mathcal{O}_{X'}(-1,0),\mathcal{O}_{X'}(0,-1)) = 1 \\
&\text{Ext}^p(\mathcal{O}_{X'}(0,-1),\mathcal{O}_{X'}(-1,0)) 
\cong \text{Ext}^{p+1}(\mathcal{O}_{X'}(-1,0),\mathcal{O}_{X'}(-1,0)) \,\, (p > 0) \\
&\text{Ext}^p(\mathcal{O}_{X'}(0,-1),\mathcal{O}_{X'}(0,-1)) 
\cong \text{Ext}^{p+1}(\mathcal{O}_{X'}(-1,0),\mathcal{O}_{X'}(0,-1))  \,\, (p > 0).
\end{split}
\end{equation}
Combining equations (\ref{coh1}), (\ref{coh1}) and (\ref{coh1}), we obtain
\[
\begin{split}
&0 = \text{Ext}^1(\mathcal{O}_{X'}(0,-1),\mathcal{O}_{X'}(0,-1)) 
\cong \text{Ext}^2(\mathcal{O}_{X'}(-1,0),\mathcal{O}_{X'}(0,-1)) \\ 
&\cong \text{Ext}^3(\mathcal{O}_{X'}(0,-1),\mathcal{O}_{X'}(0,-1)) 
\cong \text{Ext}^4(\mathcal{O}_{X'}(-1,0),\mathcal{O}_{X'}(0,-1)) \cong \dots \\
&1 = \dim \text{Ext}^1(\mathcal{O}_{X'}(-1,0),\mathcal{O}_{X'}(0,-1)) 
= \dim \text{Ext}^2(\mathcal{O}_{X'}(0,-1),\mathcal{O}_{X'}(0,-1)) \\
&= \dim \text{Ext}^3(\mathcal{O}_{X'}(-1,0),\mathcal{O}_{X'}(0,-1)) 
= \text{Ext}^4(\mathcal{O}_{X'}(0,-1),\mathcal{O}_{X'}(0,-1)) = \dots 
\end{split}
\]
hence our remaining results. 
\end{proof}

We go back to our original situation:

\begin{Cor}
\[
\begin{split}
&(1) \,\, \text{Ext}^p(L_1,L_1) = \text{Ext}^p(L_2,L_2) = 0, \,\, p > 0, p \equiv 1 (\text{mod }2). \\
&(2) \,\, \text{Ext}^p(L_1,L_2) = \text{Ext}^p(L_2,L_1) = 0, \,\, p > 0, p \equiv 0 (\text{mod }2). \\
&(3) \,\, \dim \text{Ext}^p(L_1,L_1) = \dim \text{Ext}^p(L_2,L_2) = 1, \,\, p > 0, p \equiv 0 (\text{mod }2). \\
&(4) \,\, \dim \text{Ext}^p(L_1,L_2) = \dim \text{Ext}^p(L_2,L_1) = 1, \,\, p > 0, p \equiv 1 (\text{mod }2). \\
&(5) \,\, \dim \text{Hom}(F_i,L_i) = 1 \,\, (\forall i), \,\, \text{Hom}(F_i,L_j) = 0 \,\, (i \ne j). \\
&(6) \,\, \text{Ext}^p(F_i,L_j) = 0 \,\, p > 0, \forall i, \forall j. \\
&(7) \,\, \text{Ext}^p(F_i,F_j) = 0 \,\, p > 0, \forall i, \forall j.
\end{split}
\]
\end{Cor}

\begin{proof}
Since $P \in X$ is analytically isomorphic to the situation of the lemma, we have 
isomorphisms between inner extension sheaves at the singular points.
By the spectral sequence arguments, we obtain global assertions (1) through (4).

We have exact sequences
\[
\begin{split}
&0 \to \text{Hom}(L_1,L_1) \to \text{Hom}(F_1,L_1) \to \text{Hom}(L_2,L_1) \\
&\to \text{Ext}^1(L_1,L_1) \to \text{Ext}^1(F_1,L_1) \to \text{Ext}^1(L_2,L_1) \to \dots \\
&0 \to \text{Hom}(L_1,L_2) \to \text{Hom}(F_1,L_2) \to \text{Hom}(L_2,L_2) \\
&\to \text{Ext}^1(L_1,L_2) \to \text{Ext}^1(F_1,L_2) \to \text{Ext}^1(L_2,L_2) \to \dots
\end{split}
\]
Since $\text{Hom}(L_2,L_1) = 0$, we have $\dim \text{Hom}(F_1,L_1) = 1$.
The homomorphism $\text{Hom}(L_2,L_2) \to \text{Ext}^1(L_1,L_2)$ 
is non-trivial because the extension is non-trivial.
Hence $\text{Hom}(F_1,L_2) = 0$.

We have $\text{Ext}^1(L_1,L_1) = 0$.
We have a commutative diagram
\[
\begin{CD}
\text{Ext}^p(\mathcal{O}_{X'}(-1,0),\mathcal{O}_{X'}(0,-1)) @>{\cong}>> \mathcal{E}xt^p(L_2,L_1) 
@>{\cong}>> \text{Ext}^p(L_2,L_1) \\
@VVV @VVV @VVV \\
\text{Ext}^{p+1}(\mathcal{O}_{X'}(0,-1),\mathcal{O}_{X'}(0,-1)) 
@>{\cong}>> \mathcal{E}xt^{p+1}(L_1,L_1) @>{\cong}>> \text{Ext}^{p+1}(L_1,L_1)
\end{CD}
\]
for all $p > 0$.
Therefore the homomorphisms $\text{Ext}^p(L_2,L_1) \to \text{Ext}^{p+1}(L_1,L_1)$ are bijective.
Thus we obtain $\text{Ext}^p(F_1,L_1) = 0$ for all $p > 0$.

In a similar way, we deduce that the homomorphisms 
$\text{Ext}^p(L_2,L_2) \to \text{Ext}^{p+1}(L_1,L_2)$ are bijective,
and we obtain $\text{Ext}^p(F_1,L_2) = 0$ for all $p > 0$.
The assertions for $F_2$ are obtained by symmetry.

(7) follows from exact sequences
\[
\text{Ext}^p(F_i,L_{j'}) \to \text{Ext}^p(F_i,F_j) \to \text{Ext}^p(F_i,L_j) 
\]
for all $i,j$ and $j' \ne j$.
\end{proof}

\begin{Rem}
The $2$-periodicity is a consequence of the equalities
$\bar L_1 = \bar L_2[1] = \bar L_1[2] \in D_{\text{sg}}(X)$.
\end{Rem}

We go back to the proof of the theorem.
(7) of the above corollary implies our assertion (1). 

The assertion (2) is a consequence that $F$ is an NC deformation of $L$.
Then we obtain a functor $\Phi: D^b(\text{mod-}R) \to D^b(\text{coh}(X))$ defined by 
$\Phi(\bullet) = \bullet \otimes_R F$.
We note that the functor $\Phi$ is bounded because $F$ is flat over $R$.
The $L_i$ are images of the simple $R$-modules by $\Phi$.
Since $D^b(\text{mod-}R)$ is generated by $R$, so is the image of $\Phi$ by $F$,
hence the assertion (4).

For the assertion (3), we consider 
the following exact sequences of local cohomologies:
\[
H^{p-1}_P(L_i) \to H^p_P(L_j) \to H^p_P(F_i)
\]
for $i \ne j$.
Since the $L_i$ are reflexive, we have $H^p_P(L_j) = 0$ for $p < 2$.
Since $X$ is Cohen-Macaulay, we have $H^p_P(F_i) = 0$ for $p < 3$.
Therefore we deduce that $H^p_P(L_i) = 0$ for $p < 3$, i.e., 
the $L_i$ are maximally Cohen-Macaulay sheaves. 
They generate $D_{\text{sg}}(X)$ by the condition (1).
Thus we complete the proof of the theorem.
\end{proof}

\begin{Rem}
There is an exact sequence
\[
0 \to \mathcal{O}_{\mathbf{P}^1 \times \mathbf{P}^1}(-m,0) 
\to \mathcal{O}_{\mathbf{P}^1 \times \mathbf{P}^1}^2 \to 
\mathcal{O}_{\mathbf{P}^1 \times \mathbf{P}^1}(m,0) \to 0
\]
on $\mathbf{P}^1 \times \mathbf{P}^1$ for any positive integer $m$.
But the corresponding sequence
\[
0 \to \mathcal{O}_{X_3}(-m,0) \to \mathcal{O}_{X_3}^2 \to \mathcal{O}_{X_3}(m,0) \to 0
\]
on the cone $X_3$ over $\mathbf{P}^1 \times \mathbf{P}^1$ in $\mathbf{P}^4$ is not exact if $m \ge 2$.
This follows from the fact that the fiber 
$\mathcal{O}_{X_3}(m,0) \otimes \mathcal{O}_P \cong \mathcal{O}_{X_3}(0,-m) \otimes \mathcal{O}_P$ 
at the singular point $P$ has $m+1$ generators.
Indeed if $xy+zw=0$ is the equation of $X_3$ at $P$, then
$\mathcal{O}_{X_3}(0,-m) = (x^m,x^{m-1}z,\dots,z^m)$.
Therefore the condition (1) of the theorem is necessary.
\end{Rem}

\begin{Rem}
Our construction will be generalized to higher dimensions $X_n$ with $n \ge 4$
by using spinor bundles.
Let $n' = n-1$.
On $n'$-dimensional smooth quadric $\mathbf{Q}$, 
there are locally free sheaves $\Sigma_Q$ (resp. $\Sigma^+_Q$ and $\Sigma^-_Q$) 
of rank $2^{m-1}$ for $n' = 2m-1$ (resp. rank $2^{m-1}$ for $n' = 2m$) called {\em spinor bundles}.
There are semi-orthogonal decompositions (\cite{Kapranov}):
\[
\begin{split}
&D^b(\text{coh}(Q)) = \langle \Sigma_Q(-n'),\mathcal{O}_Q(-n' + 1), 
\dots, \mathcal{O}_Q(-1),\mathcal{O}_Q \rangle, \,\,\, n'  \text{ odd} \\
&D^b(\text{coh}(Q)) = \langle \Sigma^+_Q(-n'),\Sigma^-_Q(-n'),\mathcal{O}_Q(-n' + 1), 
\dots, \mathcal{O}_Q(-1),\mathcal{O}_Q \rangle, 
\,\,\,n'  \text{ even}.
\end{split}
\]
By \cite{Ottaviani}, there are exact sequences
\[
\begin{split}
&0 \to \Sigma_Q(-1) \to \mathcal{O}_Q^{2^m} \to \Sigma_Q \to 0, \,\,\, n' = 2m-1 \\
&0 \to \Sigma^+_Q(-1) \to \mathcal{O}_Q^{2^m} \to \Sigma^-_Q \to 0, \,\,\, n' = 2m \\
&0 \to \Sigma^-_Q(-1) \to \mathcal{O}_Q^{2^m} \to \Sigma^+_Q \to 0, \,\,\, n' = 2m.
\end{split}
\]
Correspondingly, there are Cohen-Macaulay sheaves $\Sigma_X$ (resp. $\Sigma^+_X$ and 
$\Sigma^-_X$) 
of rank $2^{m-1}$ for $n = 2m$ (resp. rank $2^{m-1}$ for $n = 2m+1$) on $X = X_n$, and there are 
exact sequences
\[
\begin{split}
&0 \to \Sigma_X(-1) \to \mathcal{O}_X^{2^m} \to \Sigma_X \to 0, \,\,\, n = 2m \\
&0 \to \Sigma^+_X(-1) \to \mathcal{O}_X^{2^m} \to \Sigma^-_X \to 0, \,\,\, n = 2m+1 \\
&0 \to \Sigma^-_X(-1) \to \mathcal{O}_X^{2^m} \to \Sigma^+_X \to 0, \,\,\, n = 2m+1.
\end{split}
\]
As in the same way as in the case of $n = 2,3$ (\cite{NC} Examples 5.5 and 5.6), 
if we pull back the sequences by
injective homomorphisms $\Sigma_X(-1) \to \Sigma_X$ and 
$\Sigma_X^{\pm}(-1) \to \Sigma_X^{\pm}$, we obtain non-commutative deformations of 
$\Sigma_X(-1)$ and $\Sigma_X^{\pm}(-1)$ yielding locally free sheaves of rank $2^m$ which 
generate semi-orthogonal components of $D^b(\text{coh}(X))$.
\end{Rem}

\section{Examples}

\subsection{Quadric cone}

This is an example of a projective variety with a non-$\mathbf{Q}$-factorial ordinary double point 
considered in \cite{NC} Example 5.6.

Let $X$ be a cone over $\mathbf{P}^1 \times \mathbf{P}^1$ in $\mathbf{P}^4$ 
defined by an equation $xy+zw = 0$.
$X$ has one ordinary double point $P$, which is not $\mathbf{Q}$-factorial.
Let $\mathcal{O}_X(a,b)$ be reflexive sheaf of rank $1$ for integers $a,b$ 
corresponding to the invertible sheaf
$\mathcal{O}_{\mathbf{P}^1 \times \mathbf{P}^1}(a,b)$ of bidegree $(a,b)$.
$\mathcal{O}_X(a,b)$ is invertible if and only if $a = b$.
For example, a hyperplane section bundle is $\mathcal{O}_X(1,1)$.

Let $f: Y \to X$ be a small resolution whose exceptional locus $C$ is isomorphic to $\mathbf{P}^1$.
Let $D_1$ (resp. $D_2$) be a general member of the linear systems $\vert \mathcal{O}_X(0,1) \vert$ 
(resp. $\vert \mathcal{O}_X(1,0) \vert$), and $D'_i = f_*^{-1}D_i$ be the strict transforms.
Then we have $(D'_1,C) = 1$ and $(D'_2,C) = -1$ if $f$ was chosen suitably.
There are non-trivial extensions
\[
\begin{split}
&0 \to \mathcal{O}_X(-1,0) \to G_1 \to \mathcal{O}_X(0,-1) \to 0 \\
&0 \to \mathcal{O}_X(0,-1) \to G_2 \to \mathcal{O}_X(-1,0) \to 0 
\end{split}
\]
for some locally free sheaves $G_i$ of rank $2$.
Let $G = G_1 \oplus G_2$.
Then $G$ is a relative exceptional object over a non-commutative ring $R$ of dimension $4$ over $k$:
\[
R = \text{End}(G) = \left( \begin{matrix} k & kt \\ kt & k \end{matrix} \right) \mod t^2.
\]
The multiplication of $R$ is defined according to the matrix rule.
There is a semi-orthogonal decomposition
\[
D^b(X) = \langle \mathcal{O}(-2,-2), \mathcal{O}(-1,-1), G, \mathcal{O} \rangle 
\cong \langle D^b(k), D^b(k), D^b(R), D^b(k) \rangle.
\]

\subsection{2 point blow up of $\mathbf{P}^3$}

This is another example of a projective variety with a non-$\mathbf{Q}$-factorial ordinary double point. 

Let $g: Y \to \mathbf{P}^3$ be a blowing up at 2 general points $P_1,P_2$, 
with exceptional divisors $E_1,E_2$.
Let $l$ be the strict transform of the line connecting $P_1,P_2$.
Let $f: Y \to X$ be the contraction of $l$ to a point.

Let $H$ be the strict transform of a general plane on $\mathbf{P}^3$ to $Y$.
We have $K_Y = -4H + 2E_1 + 2E_2$, and
$f$ is given by the anti-canonical linear system $\vert -K_Y \vert$ which is nef and big.
By the contraction theorem (\cite{KMM}), if $(D,l) = 0$ for a Cartier divisor $D$ on $Y$, 
then $\mathcal{O}_Y(D) = f^*\mathcal{O}_X(f_*D)$ for 
a Cartier divisor $f_*D$ on $X$.
By \cite{Beilinson} and \cite{BO}, 
$D^b(Y)$ is classically generated by a full exceptional collection 
\[
D^b(Y) = \langle \mathcal{O}_{E_1}(2E_1), \mathcal{O}_{E_2}(2E_1), 
\mathcal{O}_{E_1}(E_1), \mathcal{O}_{E_2}(E_2),
\mathcal{O}(-3H), \mathcal{O}(-2H), \mathcal{O}(-H), O \rangle.
\] 

The following lemma shows that the conditions of Theorem~\ref{ODP} are satisfied:

\begin{Lem}
Let $D_1 = -H + E_1+E_2$, $D_2 = -E_1$, and $L_i = f_*\mathcal{O}_Y(-D_i)$ for $i=1,2$.
Then the following hold.

(1) $(D_1,l) = 1$, $(D_2,l) = -1$ and $R^pf_*\mathcal{O}_Y(D_i) = 0$ for $p > 0$ and all $i$.

(2) $(L_1,L_2)$ is a simple collection.

(3) $H^p(X, f_*(D_i-D_j)) = 0$ for all $p > 0$ and all $i,j$.
\end{Lem}

\begin{proof}
(1) is clear.
(2) Since the $L_i$ are reflexive sheaves of rank $1$, we have $\dim \text{Hom}(L_i,L_i) = 1$.
We have $\text{Hom}(L_1,L_2) = H^0(Y,D_1-D_2) = H^0(Y,-H+2E_1+E_2) = 0$. 
We have also $\text{Hom}(L_2,L_1) = H^0(Y,-D_1+D_2) = H^0(Y,H-2E_1-E_2) = 0$.  

(3) If $i = j$, then $H^p(X,f_*O)) = H^p(Y,O) = 0$ for $p > 0$. 
We consider the case $i \ne j$.
There is a commutative diagram of exact sequences
\[
\begin{CD}
0 @>>> \mathcal{O}_Y(-D_1+D_2) @>>> \mathcal{O}_Y(H) @>>> \mathcal{O}_{2E_1} 
\oplus \mathcal{O}_{E_2} @>>> 0 \\
@. @VVV @VVV @VVV  @. \\
0 @>>> \mathcal{O}_l(-D_1+D_2) @>>> \mathcal{O}_l(H) @>>> \mathcal{O}_{2Q_1} 
\oplus \mathcal{O}_{Q_2} @>>> 0 
\end{CD}
\]
where $Q_i = E_i \cap l$.
In the associated long exact sequences, we have 
$H^0(\mathcal{O}_Y(-D_1+D_2)) = H^0(\mathcal{O}_l(-D_1+D_2)) = 0$, 
$\dim H^0(Y,\mathcal{O}(H)) = 4$, $\dim H^0(l,\mathcal{O}_l(H)) = 2$, 
$\dim H^0(\mathcal{O}_{2E_1} \oplus \mathcal{O}_{E_2}) = 5$, 
$\dim H^0(\mathcal{O}_{2Q_1} \oplus \mathcal{O}_{Q_2}) = 3$, and
$H^p(Y,\mathcal{O}(H)) = H^p(l,\mathcal{O}_l(H)) = 0$ for $p > 0$.
It follows that $\dim H^1(Y,\mathcal{O}(-D_1+D_2)) = \dim H^1(\mathcal{O}_l(-D_1+D_2)) = 1$ and 
$H^p(Y,\mathcal{O}(-D_1+D_2)) = 0$ for $p \ne 1$. 
Moreover the homomorphisms $H^0(Y,\mathcal{O}(H)) \to H^0(l,\mathcal{O}_l(H))$ and 
$H^0(\mathcal{O}_{2E_1} \oplus \mathcal{O}_{E_2}) 
\to H^0(\mathcal{O}_{2Q_1} \oplus \mathcal{O}_{Q_2})$
are surjective.
It follows that the homomorphism $H^1(Y,\mathcal{O}_Y(-D_1+D_2)) 
\to H^1(\mathcal{O}_l(-D_1+D_2))$ is also surjective.

We have an exact sequence
\[
\begin{split}
&0 \to H^1(X,f_*\mathcal{O}_Y(-D_1+D_2)) \to H^1(Y,\mathcal{O}_Y(-D_1+D_2)) 
\to H^0(X,R^1f_*\mathcal{O}_Y(-D_1+D_2)) \\
&\to H^2(X,f_*\mathcal{O}_Y(-D_1+D_2)) \to H^2(Y,\mathcal{O}_Y(-D_1+D_2)).
\end{split}
\]
Since the restriction homomorphism $H^1(Y,\mathcal{O}_Y(-D_1+D_2)) 
\to H^1(l,\mathcal{O}_l(-D_1+D_2))$ is surjective, 
we conclude that $H^p(X,f_*(-D_1+D_2)) = 0$ for $p > 0$.

There is an exact sequence
\[
0 \to \mathcal{O}(-H) \to \mathcal{O}(D_1-D_2) 
\to \mathcal{O}_{2E_1}(2E_1) \oplus \mathcal{O}_{E_2}(E_2) \to 0.
\]
We have 
\[
H^p(Y,\mathcal{O}(-H)) = H^p(\mathcal{O}_{2E_1}(2E_1) \oplus \mathcal{O}_{E_2}(E_2)) = 0
\]
for all $p$.
Hence $H^p(Y,\mathcal{O}_Y(D_1-D_2)) = 0$ for all $p$.
Since $R^pf_*\mathcal{O}_Y(D_1-D_2) = 0$ for $p > 0$, 
we conclude that $H^p(X,f_*\mathcal{O}_Y(D_1-D_2)) = 0$ for $p > 0$.
\end{proof}

Let 
\[
\begin{split}
0 \to L_2 \to F_1 \to L_1 \to 0 \\
0 \to L_1 \to F_2 \to L_2 \to 0 
\end{split}
\]
be the universal extensions corresponding to $\text{Ext}^1(L_i,L_j)$ for $i \ne j$.
Let $F = F_1 \oplus F_2$.
We will calculate the right orthogonal complement $F^{\perp}$ in the rest of the example.

We denote 
\[
\begin{split}
&C'_1 = -3H+2E_1+E_2, \,\, C'_2 = -3H+E_1+2E_2 \\ 
&C'_3 = -2H+E_1+E_2, \,\, C'_4 = -H+E_1, \,\, C'_5 = 0.
\end{split}
\]
Then $(C'_i,l) = 0$ for all $i$.
Let $C_i = f_*C'_i$ be Cartier divisors on $X$ such that $C'_i = f^*C_i$.

\begin{Lem}
The right orthogonal complement $F^{\perp}$ in $D^b(\text{coh}(X))$ is generated by 
a strong exceptional collection consisting of the invertible sheaves
$(\mathcal{O}_X(C_1),\dots,\mathcal{O}_X(C_5))$.
\end{Lem}

\begin{proof}
(1) We prove that the sequence is a strong exceptional collection.

Since $H^p(X,\mathcal{O}_X) = 0$ for $p > 0$, the $\mathcal{O}_X(C_i)$ are exceptional objects.
We prove that $\text{Hom}(\mathcal{O}(C_i),\mathcal{O}(C_j)[p]) = 0$ for $i > j$ and for all $p$, 
and $\text{Hom}(\mathcal{O}(C_i),\mathcal{O}(C_j)[p]) = 0$ for all $i,j$ and $p > 0$.
We note that $\text{Hom}_X(\mathcal{O}_X(C_i),\mathcal{O}_X(C_j)[p]) 
\cong \text{Hom}_Y(\mathcal{O}_Y(C'_i),\mathcal{O}_Y(C'_j)[p])$.
We calculate 
\[
\begin{split}
&H^p(Y,E_1-E_2) = H^p(Y,-H+E_1) = H^p(Y,-H+E_2) = H^p(Y,-2H+E_1+E_2) \\
&=H^p(Y,-2H+2E_2) = H^p(Y,-3H+2E_1+E_2) = H^p(Y,-3H+E_1+2E_2) = 0
\end{split}
\]
for all $p$.
For example, we have an exact sequence 
\[
\dots \to H^p(Y,E_1-E_2) \to H^p(Y,E_1) \to H^p(E_2,\mathcal{O}_{E_2}) \to \dots.
\]
Since $H^p(Y,E_1) \cong H^p(Y,\mathcal{O}_Y)$, we have $H^p(Y,E_1-E_2)=0$ for all $p$.
We have $H^p(Y,-3H+E_1+2E_2) = H^p(Y,-3H) = 0$ for all $p$, etc.
We also calculate 
\[
\begin{split}
&H^p(Y,-E_1+E_2) = H^p(Y,H-E_1) = H^p(Y,H-E_2) = H^p(Y,2H-E_1-E_2) \\
&= H^p(Y,2H-2E_2) = H^p(Y,3H-2E_1-E_2) = H^p(Y,3H-E_1-2E_2) = 0
\end{split}
\]
for $p > 0$.
For example, we have an exact sequence 
\[
\dots \to H^p(Y,H-E_1) \to H^p(Y,H) \to H^p(E_1,\mathcal{O}_{E_1}) \to \dots.
\]
Since $H^0(Y,H) \to H^0(E_1,\mathcal{O}_{E_1})$ is surjective and $H^p(Y,H)=0$ for $p > 0$, we have
$H^p(Y,H-E_1) = 0$ for $p > 0$, etc.

\vskip 1pc

(2) We prove that the $\mathcal{O}(C_i)$ belong to $F^{\perp}$.

By the Grothendieck duality, we have 
\[
\text{Hom}_X(L_i,\mathcal{O}(C_j)[p]) = \text{Hom}_X(Rf_*\mathcal{O}_Y(D_i),\mathcal{O}_X(C_j)[p]) 
\cong \text{Hom}_Y(\mathcal{O}_Y(D_i),\mathcal{O}_Y(C'_j)[p])
\]
because $f^!\mathcal{O}_X(C_j) \cong \mathcal{O}_Y(C'_j)$. 
Therefore we will prove that 
$H^p(Y,C_j-H+E_1+E_2) = H^p(Y,C_j - E_1) = 0$ for all $j$ and $p$.

Since $K_Y = -4H+2E_1+2E_2$, we have
\[
H^p(Y,-4H+3E_1+2E_2) = H^{3-p}(Y,-E_1)^* = 0
\]
for all $p$.
We also have 
\[
H^p(Y,-3H+E_1+E_2) \cong H^p(\mathbf{P}^3,-3H) = 0 
\]
for all $p$.
Therefore we have $\mathcal{O}(C_1) \in F^{\perp}$.
 
We have
\[
\begin{split}
&H^p(Y,-4H+2E_1+3E_2) = H^{3-p}(Y,-E_2)^* = 0 \\
&H^p(Y,-3H+2E_2) \cong H^p(\mathbf{P}^3,-3H) = 0 
\end{split}
\]
for all $p$, hence we have $\mathcal{O}(C_2) \in F^{\perp}$.

We have
\[
\begin{split}
&H^p(Y,\mathcal{O}(-3H+2E_1+2E_2)) \cong H^p(\mathbf{P}^3,\mathcal{O}(-3H)) = 0 \\
&H^p(Y,\mathcal{O}(-2H+E_2)) \cong H^p(\mathbf{P}^3,\mathcal{O}(-3H)) = 0 \\ 
&H^p(Y,\mathcal{O}(-2H+2E_1+E_2)) \cong H^p(\mathbf{P}^3,\mathcal{O}(-2H)) = 0 \\
&H^p(Y,\mathcal{O}(-H)) = 0 \\ 
&H^p(Y,\mathcal{O}(-H+E_1+E_2)) \cong H^p(\mathbf{P}^3,\mathcal{O}(-H)) = 0 \\
&H^p(Y,\mathcal{O}(-E_1)) = 0 
\end{split}
\]
for all $p$, hence we have $\mathcal{O}_X(C_i) \in F^{\perp}$ for $i=3,4,5$.

\vskip 1pc

(3) We prove that the $\mathcal{O}_X(C_i)$ and the $L_j$ generate $D^b(\text{coh}(X))$.
Then it follows that the $\mathcal{O}_X(C_i)$ generates $F^{\perp}$.
We denote by $C$ the triangulated subcategory of $D^b(\text{coh}(X))$ 
generated by the $\mathcal{O}_X(C_i)$ and the $L_j$.

The linear system $\vert H - E_1 - E_2 \vert$ is a pencil, and its base locus is nothing but $l$.
The image of the the natural homomorphism $\mathcal{O}_Y^2 \to \mathcal{O}_Y(H - E_1 - E_2)$ 
is equal to $I_l(H - E_1 - E_2)$, 
where $I_l$ is the ideal sheaf of $l$.
Hence we have an exact sequence
\[
0 \to \mathcal{O}_Y(-H+E_1+E_2) \to \mathcal{O}_Y^2 \to \mathcal{O}_Y(H - E_1 - E_2) 
\to \mathcal{O}_l(-1) \to 0.
\]
By pushing down to $X$, we obtain an exact sequence
\[
0 \to f_*\mathcal{O}_Y(-H+E_1+E_2) \to \mathcal{O}_X^2 \to L_1 \to 0.
\]
Therefore $C$ coincides with the triangulated subcategory generated by the 
$\mathcal{O}_X(C_j)$, the $L_j$ and 
$f_*\mathcal{O}_Y(-H+E_1+E_2)$.

We have exact sequences
\[
\begin{split}
&0 \to \mathcal{O}_X \to f_*\mathcal{O}_Y(E_1) \to f_*\mathcal{O}_{E_1}(E_1) \to 0 \\
&0 \to f_*\mathcal{O}_Y(-H+E_2) \to f_*\mathcal{O}_Y(-H+E_1+E_2) \to f_*\mathcal{O}_{E_1}(E_1) \to 0.
\end{split}
\]
Thus $f_*\mathcal{O}_Y(-H+E_2)$ can also be included in the set of generators of $C$.
We note that $(-H+E_2,l) = 0$, hence $f_*\mathcal{O}_Y(-H+E_2)$ is an invertible sheaf on $X$.

We need some lemmas:

\begin{Lem}\label{classical}
$D^b(\text{coh}(Y))$ is classically generated by the following full exceptional collection:
\[
\begin{split}
D^b(\text{coh}(Y)) = \langle &\mathcal{O}_Y(-3H+2E_1+E_2), \mathcal{O}_Y(-3H+E_1+2E_2), 
\mathcal{O}_Y(-2H+E_1+E_2), \\
&\mathcal{O}_Y(-H+E_1), \mathcal{O}_Y(-H+E_2)), \mathcal{O}_Y(-H+E_1+E_2), \mathcal{O}_Y, 
\mathcal{O}_Y(H-E_1-E_2) \rangle.
\end{split}
\]
\end{Lem}

We note that the images by $Rf_*$ of these exceptional objects on $Y$ are exactly the objects
considered above as the generators of $C$.

\begin{proof}
We first prove that these objects constitute an exceptional collection.
Since they are all line bundles, they are exceptional objects.
We check their semi-orthogonality.
We have $H^p(Y,E_1-E_2) = H^p(Y,-H+E_i) = H^p(Y,-2H+E_1+E_2)=H^p(Y,-2H+2E_i)=0$ for all $p$ 
and all $i$, 
hence the first 5 are semi-orthogonal.
We have $H^p(Y,-2H+2E_1+2E_2) = H^p(-H+E_1+E_2) = 0$ for all $p$, 
hence the latter 3 are also semi-orthogonal.

We have $H^p(Y,-2H+E_i)=H^p(Y,-H)=H^p(Y,-E_i)=0$ and 
$H^p(Y,-3H+2E_i+E_j)=H^p(Y,-2H+2E_i+E_j)=0$ for all $p$ and $i \ne j$.
By the Serre duality, $H^p(Y,-4H+3E_i+2E_j)$ is dual to $H^{3-p}(Y,-E_i)=0$ for $i \ne j$.
Hence the first 5 and the latter 3 are semi-orthogonal, and these 8 objects make an exceptional collection.

We prove that these objects classically generate $D^b(\text{coh}(Y))$.
Let $T$ be the full triangulated subcategory of $D^b(\text{coh}(Y))$ 
classically generated by the above exceptional collection.
By the exact sequences
\[
\begin{split}
&0 \to \mathcal{O}_Y(-H+E_i) \to \mathcal{O}_Y(-H+E_1+E_2) \to \mathcal{O}_{E_j}(E_j) \to 0 \\
&0 \to \mathcal{O}_Y(-H) \to \mathcal{O}_Y(-H+E_i) \to \mathcal{O}_{E_i}(E_i) \to 0 \\
&0 \to \mathcal{O}_Y(-2H) \to \mathcal{O}_Y(-2H+E_1+E_2) 
\to \mathcal{O}_{E_1}(E_1) \oplus \mathcal{O}_{E_2}(E_2) \to 0
\end{split}
\]
for $i \ne j$, we deduce that the objects $\mathcal{O}_{E_i}(E_i)$ for $i=1,2$, $\mathcal{O}_Y(-H)$ and 
$\mathcal{O}_Y(-2H)$ can be included 
in the set of classical generators of $T$.

From the exact sequences
\[
\begin{split}
&0 \to \mathcal{O}(-H+E_1+E_2) \to O^2 \to \mathcal{O}(H - E_1 - E_2) \to \mathcal{O}_l(-1) \to 0 \\
&0 \to \mathcal{O}(-3H+2E_1+2E_2) \to \mathcal{O}(-2H+E_1+E_2)^2 \to \mathcal{O}(-H) 
\to \mathcal{O}_L(-1)\to 0
\end{split}
\]
we deduce that $\mathcal{O}_Y(-3H+2E_1+2E_2)$ can also be included.
From an exact sequence
\[
0 \to \mathcal{O}_Y(-3H+E_i+2E_j) \to \mathcal{O}_Y(-3H+2E_1+2E_2) \to \mathcal{O}_{E_i}(2E_i) \to 0
\]
for $i \ne j$, we deduce that $\mathcal{O}_{E_i}(2E_i)$ can be included for $i = 1,2$.
From
\[
0 \to \mathcal{O}_Y(-3H) \to \mathcal{O}_Y(-3H+2E_1+2E_2) 
\to \mathcal{O}_{2E_1}(2E_1) \oplus \mathcal{O}_{2E_2}(2E_2) \to 0
\]
we deduce that $\mathcal{O}_Y(-3H)$ can be included.
Therefore $T = D^b(\text{coh}(Y))$.
\end{proof}

The following lemma says that the direct image functor
$Rf_*$ for a birational morphism $f$ 
is almost surjective for the bounded derived categories of coherent sheaves if $Rf_*$ 
preserves the structure sheaves:

\begin{Lem}
Let $f: Y \to X$ be a birational morphism of projective varieties.
Assume that $Rf_*\mathcal{O}_Y \cong \mathcal{O}_X$.
Then the Karoubian envelope of the image of the functor $Rf_*: D^b(\text{coh}(Y)) \to D^b(\text{coh}(X))$ 
coincides with $D^b(\text{coh}(X))$.
\end{Lem}

\begin{proof}
By the assumption, we have $Rf_*Lf^* \cong \text{Id}$ on $D(\text{Qcoh}(X))$.
Let $A \in D^b(\text{coh}(X))$ be any object.
Then we have $Lf^*A \in D^-(\text{coh}(Y))$.
We take a large integer $m$ and consider a natural distinguished triangle for truncations: 
\[
\begin{CD}
\tau_{<-m}Lf^*A @>h>> Lf^*A @>>> \tau_{\ge -m}Lf^*A @>>> (\tau_{<-m}Lf^*A)[1].
\end{CD}
\]
We have a morphism $Rf_*(h): Rf_*(\tau_{<-m}Lf^*A) \to Rf_*Lf^*A \cong A$.
Since $Rf_*$ is bounded and $A$ is bounded, $Rf_*(h) = 0$ for sufficiently large $m$. 
It follows that $A$ is a direct summand of $Rf_*(\tau_{\ge -m}Lf^*A)$ in $D(\text{Qcoh}(X))$ because
$D(\text{Qcoh}(X))$ is Karoubian by \cite{BN}:
\[
Rf_*(\tau_{\ge -m}Lf^*A) \cong Rf_*(\tau_{<-m}Lf^*A)[1] \oplus A.
\]
Since $Rf_*(\tau_{\ge -m}Lf^*A) \in D^b(\text{coh}(X))$, we conclude that $A$ belongs 
to the Karoubian completion
of the image of $Rf_*$.
\end{proof}

Let $G$ be the set of exceptional objects which classically generates 
$D^b(\text{coh}(Y))$ in Lemma~\ref{classical}.
We will prove that $Rf_*G$ generates $D^b(\text{coh}(X))$.
Let $A \in D^b(\text{coh}(X))$ be any object.
We need to prove that, if $\text{Hom}(Rf_*G,A[p]) = 0$ for all $p$, then $A \cong 0$.
 
We know that $A$ is a direct summand of $Rf_*B$ for some object $B \in D^b(\text{coh}(Y))$.
Since $G$ classically generates $D^b(\text{coh}(Y))$, we deduce that $\text{Hom}(Rf_*B,A[p]) = 0$ 
for all $B$ and all $p$.
Then it follows that $A \cong 0$.
\end{proof}

\begin{Rem}
It is interesting to calculate the derived categories of varieties which are obtained by blowing up 
$\mathbf{P}^3$ at 
more than $2$ points.
Especially, if we blow up $6$ or more points, then the blown-up varieties have moduli.
It is interesting to determine the semi-orthogonal complement of the trivial factors in this case.
\end{Rem}

\subsection{Locally factorial case}

We consider an example of a $3$-fold with a $\mathbf{Q}$-factorial ordinary double point.
We will see that similar arguments to the non-$\mathbf{Q}$-factorial case 
do not work because NC deformations do not terminate.

We start with an example of a singular curve: 

\begin{Expl}
Let $X$ be a nodal cubic curve defined by an equation 
$(x^2+y^2)z+y^3=0$ in $\mathbf{P}^2$.

Let $P \in X$ be the singular point and let $\nu: X' \to X$ be the normalization.
Then $X' \cong \mathbf{P}^1$ and $H^1(\mathcal{O}_X) = k$.
We consider non-commutative deformations of a Cohen-Macaulay sheaf 
$L = \nu_*\mathcal{O}_{X'}$ which generates $D_{\text{sg}}(X)$.

From an exact sequence 
\[
0 \to \mathcal{O}_X \to L \to \mathcal{O}_P \to 0, 
\]
we consider an associated long exact sequence to obtain
\[
\begin{split}
&\mathcal{H}om_{\mathcal{O}_X}(L,L) \cong  \mathcal{H}om_{\mathcal{O}_X}(\mathcal{O}_X,L) \cong L \\
&\mathcal{E}xt^p_{\mathcal{O}_X}(\mathcal{O}_P,L) \cong \mathcal{E}xt^p_{\mathcal{O}_X}(L,L), \,\, p > 0.
\end{split}
\]
Since $X$ is Gorenstein, we have 
\[
\mathcal{E}xt^p_{\mathcal{O}_X}(\mathcal{O}_P,\mathcal{O}_X) 
\cong \begin{cases} \mathcal{O}_P, \,\,&p = 1\\ 
0, &p \ne 1. \end{cases}
\]
Therefore from another associated long exact sequence, we obtain
\[
\mathcal{E}xt^p_{\mathcal{O}_X}(\mathcal{O}_P,L) 
\cong \mathcal{E}xt^p_{\mathcal{O}_X}(\mathcal{O}_P,\mathcal{O}_P)
\]
for all $p > 0$.
In particular we have 
$\mathcal{E}xt^1_{\mathcal{O}_X}(L,L) 
\cong \mathcal{E}xt^1_{\mathcal{O}_X}(\mathcal{O}_P,\mathcal{O}_P) \cong k^2$.

The versal NC deformation of $\mathcal{O}_P$ on $X$ is given by the completion of $X$ at $P$.
Its parameter algebra is given by $k[[x,y]]/(xy)$.
The versal NC deformation of $L$ is given by an infinite chain of smooth rational curves. 
It is the inverse limit of the sheaves $L_{i,j}$ for $i,j \to \infty$ defined in the following way.
$L_{i,j}$ is the direct image sheaf of an invertible sheaf $L'_{i,j}$ 
on a chain of smooth rational curves of type $A_{i+j+1}$, where the degree of $L_{i,j}$ 
on the $m$-th component 
is equal to $0$ for $m = i + 1$, and to $1$ otherwise. 
The parameter algebra for $L_{i,j}$ is given by $k[x,y]/(xy,x^{i+1},y^{j+1})$.
In particular NC deformations of $L$ do not stop after finitely many steps.
\end{Expl}

We consider a locally factorial surface:

\begin{Expl}\label{K3}
Let $\pi: Y_2 \to \mathbf{P}^2$ be a double cover whose ramification divisor is a generic curve of 
degree $6$
with one node at $Q \in \mathbf{P}^2$.
Then $P = \pi^{-1}(Q) \in Y_2$ is the only singular point of $Y_2$.

We claim that $P \in Y_2$ is a factorial ordinary double point.
This construction and the following argument is communicated by Keiji Oguiso.
Let $Y' \to Y_2$ be a minimal resolution with an exceptional divisor $E$.
Then $Y'$ is a K3 surface, and the Neron-Severi lattice is given by 
$NS(Y') = \mathbf{Z}H \oplus \mathbf{Z}E$, 
where $H$ is the pull-back of a line on $\mathbf{P}^2$, 
due to the genericity of the ramification divisor.
We have $(H^2) = 2$, $(E^2) = -2$, and $(H,E) = 0$, hence $P \in Y_2$ is factorial.

We note that an ordinary double point on a rational surface, 
say $S$, is always non-factorial (though $2$-factorial).
Indeed the Neron-Severi lattice of its resolution $S' \to S$ is always unimodular 
since $H^2(\mathcal{O}_{S'}) = 0$.
Hence there is a curve on $S'$ whose intersection number with the exceptional curve is odd, 
and its image on $S$
is not a Cartier divisor.  

Let $l$ be a generic line in $\mathbf{P}^2$ through $Q$, and let $C = \pi^{-1}(l)$.
Then $C$ is an irreducible curve of genus $1$ with a node.
Let $\nu: C' \to C$ be the normalization, and let $L_{C'}$ be an invertible sheaf
on $C'$ of degree $2$.
Then there is a surjective homomorphism $\mathcal{O}_C^{\oplus 2} \to \nu_*L_{C'}$.
Let $L$ be the kernel of the composite homomorphism 
$\mathcal{O}_{Y_2}^{\oplus 2} \to \mathcal{O}_C^{\oplus 2} \to \nu_*L_{C'}$:
\[
0 \to L \to \mathcal{O}_{Y_2}^{\oplus 2} \to \nu_*L_{C'} \to 0.
\]
$L$ is a reflexive sheaf of rank $2$ which is locally free except at $P$.
We obtain $H^p(L) = 0$ for all $p$ from a long exact sequence.
$C$ has two analytic branches at $P$, 
and $L$ is analytically isomorphic to a direct sum of reflexive sheaves of rank $1$ near $P$.

We have exact sequences
\[
\begin{split}
&0 \to L^{\oplus 2} \to \mathcal{H}om(L,L) \to \mathcal{E}xt^1(\nu_*L_{C'},L) \to 0 \\
&0 \to \mathcal{H}om(\nu_*L_{C'},\nu_*L_{C'}) \to \mathcal{E}xt^1(\nu_*L_{C'},L)
\to \mathcal{E}xt^1(\nu_*L_{C'},\mathcal{O}_{Y_2}^{\oplus 2}).
\end{split}
\]
Since $Y_2$ is Gorenstein with $\omega_{Y_2} \cong \mathcal{O}_{Y_2}$, 
we have by the Grothendieck duality 
\[
\mathcal{E}xt^1(\nu_*L_{C'},\mathcal{O}_{Y_2}) \cong \nu_*L_{C'}^{-1}.
\]
Since $H^0(\nu_*L_{C'}^{-1}) = 0$, we have 
\[
\text{Hom}(L,L) \cong H^0(\mathcal{E}xt^1(\nu_*L_{C'},L)) \cong \text{Hom}(\nu_*L_{C'},\nu_*L_{C'}) 
\cong k.
\]
Thus $L$ is a simple sheaf on $Y_2$.

The NC deformations of $L$ do not stop after finitely many steps.
Indeed the two analytic components extend in an infinite chain as in the previous example. 
\end{Expl}

Now we consider a $3$-fold with a $\mathbf{Q}$-factorial, hence factorial, ordinary double point:

\begin{Expl}
Let $X_0$ be a cubic $3$-fold in $\mathbf{P}^4$ with coordinates $(x,y,z,w,t)$ defined by an equation
\[
x^3+3xy^2+w^3-3t(xy+z^2+w^2) = 0.
\]
The singular locus of $X_0$ consists of $2$-points $P = (0,0,0,0,1)$ and $P' = (0,-1,0,0,1)$ 
which are ordinary double points.
Let $g: X \to X_0$ be the blowing up at $P'$ with the exceptional divisor 
$E \cong \mathbf{P}^1 \times \mathbf{P}^1$.

The local equation of $X$ at $P$ can be written as
\[
x(x^2+3y^2+y) + z^2 + w^2 + w^3 = 0.
\]
Hence $X$ is a projective variety with one $\mathbf{Q}$-factorial ODP.

Let $D_0$ be a prime divisor on $X_0$ defined by $x = 0$.
Then $D_0$ has an equation
\[
w^3-3t(z^2+w^2) = 0
\]
in $\mathbf{P}^3$ with coordinates $(y,z,w,t)$.
$D_0$ is a cone over a nodal cubic curve.
The vertex of the cone is at $Q = (0,1,0,0,0)$.
The singular locus of $D_0$ is a line defined by $z = w = 0$.
Let $\nu_0: D'_0 \to D_0$ be the normalization.
$D'_0$ is the cone over a normal rational curve of degree $3$.

$g$ induces a blowing up $g_D: D \to D_0$ at $P'$.
The exceptional locus of $g_D$ consists of two lines $m_1, m_2$ on $\mathbf{P}^1 \times \mathbf{P}^1$.
The singular locus of $D$ is a smooth rational curve which is the strict transform of the line 
$\{z = w = 0\}$. 
Let $\nu: D' \to D$ be the normalization.

Let $l$ be a generic line on $D'$ through the vertex.
Then there is a surjective homomorphism $\mathcal{O}_D^{\oplus 2} \to \nu_*\mathcal{O}_{D'}(l)$.
Let $L$ be the kernel of the composition 
$\mathcal{O}_X^{\oplus 2} \to \mathcal{O}_D^{\oplus 2} \to \nu_*\mathcal{O}_{D'}(l)$. 
Thus we have an exact sequence
\[
0 \to L \to \mathcal{O}_X^{\oplus 2} \to \mu_*\mathcal{O}_{D'}(l) \to 0.
\]
There is an exact sequence
\[
\dots \to H^p_P(L) \to H^p_P(\mathcal{O}_X^2) \to H^p_P(\nu_*\mathcal{O}_{D'}(l)) \to \dots
\] 
Since $\mathcal{O}_X$ and $\nu_*\mathcal{O}_{D'}(l)$ have depth $3$ and $2$, respectively, 
$L$ is a maximally Cohen-Macaulay 
sheaf of rank $2$ on $X$ which is locally free except at $P$.
$L$ generates $D_{\text{sg}}(X)$.

We will prove that $L$ is a simple sheaf, i.e., $\text{End}(L) \cong k$.
We have $\dim H^0(\mathcal{O}_X) = 1$, $H^p(\mathcal{O}_X) = 0$ for $p > 0$, 
$\dim H^0(\mathcal{O}_{D'}(l)) = 2$ and 
$H^p(\mathcal{O}_{D'}(l)) = 0$ for $p > 0$.
Therefore we have $H^p(L) = 0$ for all $p$.
We have an exact sequence
\[
0 \to L^{\oplus 2} \to \mathcal{H}om(L,L) \to \mathcal{E}xt^1(\mu_*\mathcal{O}_{D'}(l),L) \to 0
\]
and $\mathcal{E}xt^p(L,L) \cong \mathcal{E}xt^{p+1}(\mu_*\mathcal{O}_{D'}(l),L)$ for $p > 0$.

Since $X$ is Gorenstein, we have 
\[
\mathcal{E}xt^p(\mu_*\mathcal{O}_{D'}(l),\mathcal{O}_X) \cong \begin{cases} \mu_*\omega_{D'/X}(-l),
\,\, &p = 1 \\
0, \,\, &p \ne 1 \end{cases}
\]
by the Grothendieck duality.
From a long exact sequence, we deduce
\[
\begin{split}
&0 \to \mathcal{H}om(\mu_*\mathcal{O}_{D'}(l),\mu_*\mathcal{O}_{D'}(l)) 
\to \mathcal{E}xt^1(\mu_*\mathcal{O}_{D'}(l),L) 
\to \mu_*\omega_{D'/X}(-l)^{\oplus 2} \\
&\to \mathcal{E}xt^1(\mu_*\mathcal{O}_{D'}(l),\mu_*\mathcal{O}_{D'}(l)) 
\to \mathcal{E}xt^2(\mu_*\mathcal{O}_{D'}(l),L) \to 0
\end{split}
\]
and 
$\mathcal{E}xt^p(\mu_*\mathcal{O}_{D'}(l),\mu_*\mathcal{O}_{D'}(l)) 
\cong \mathcal{E}xt^{p+1}(\mu_*\mathcal{O}_{D'}(l),L)$ for $p > 1$.

Since $D \sim \mathcal{O}_X(1) - E$ on $X$, we calculate 
\[
\omega_{D'/X}(-l) \cong \mathcal{O}_{D'}(3l - (m_1+m_2) - (l-m_1) - (l-m_2) - l) \cong \mathcal{O}_{D'}.
\]
Thus $H^0(\mu_*\omega_{D'/X}(-l)^{\oplus 2}) \cong k^2$.
$L$ is a locally free sheaf outside $P$, 
and $\mu_*\mathcal{O}_{D'}(l)$ is an invertible sheaf on 
the smooth locus of a Cartier divisor $D$.
Moreover $\mu_*\mathcal{O}_{D'}(l)$ is analytically isomorphic to the sum of invertible sheaves 
on two Cartier divisors 
along the double locus of $D$ except at $P$ and $Q$.
Therefore we have $\mathcal{E}xt^2(\mu_*\mathcal{O}_{D'}(l),L)$ is supported at $\{P,Q\}$.
On the other hand, $\mathcal{E}xt^1(\mu_*\mathcal{O}_{D'}(l),\mu_*\mathcal{O}_{D'}(l))$ 
is an invertible sheaf on the smooth locus of $D$
and has higher rank along the double locus of $D$.
It follows that the homomorphism
\[
H^0(\mu_*\omega_{D'/X}(-l)^{\oplus 2}) 
\to H^0(\mathcal{E}xt^1(\mu_*\mathcal{O}_{D'}(l),\mu_*\mathcal{O}_{D'}(l)))
\]
is injective.
Therefore 
\[
H^0(\mathcal{E}xt^1(\mu_*\mathcal{O}_{D'}(l),L)) 
\cong H^0(\mathcal{H}om(\mu_*\mathcal{O}_{D'}(l),\mu_*\mathcal{O}_{D'}(l))) \cong k
\]
hence $L$ is a simple sheaf.

We consider NC deformations of $L$.
The successive extensions of $L$ become an infinite chain as in the previous examples 
so that they do not stop after finitely many steps.
Thus we do not obtain an SOD unlike the case of non-Q-factorial ODP. 

We note that there is an extension $F$ of $L$ by $L$ which is locally free.
But $F$ has still an extension by $L$, and the successive extensions by $L$ do not stop.
\end{Expl}

\section{Appendix: Correction to \cite{NC}}

We correct an error in \cite{NC}.

In Example 5.7 of \cite{NC}, 
it is claimed that a collection $(\mathcal{O}_X(-d), G, \mathcal{O}_X)$ yields 
a semi-orthogonal decomposition of $D^b(\text{coh}(X))$.
But it is not the case because the semi-orthogonality fails: $R\text{Hom}(G,\mathcal{O}_X(-d)) \ne 0$.
The correct collection is $(G(-d), \mathcal{O}_X(-d), \mathcal{O}_X)$.
Then we have $D^b(\text{coh}(X)) \cong \langle D^b(R), D^b(k), D^b(k) \rangle$.  

The only thing which is left to be proved is the semi-orthogonality.
But we have $R\text{Hom}(\mathcal{O}_X(-d),G(-d)) \cong RH(X,G) = 0$.

There are more information in \cite{Kuznetsov} and \cite{KKS}.


Graduate School of Mathematical Sciences, University of Tokyo,
Komaba, Meguro, Tokyo, 153-8914, Japan. 

Department of Mathematical Sciences, 
KAIST, 291 Daehak-ro, Yuseong-gu, Daejeon 34141, Korea.

National Center for Theoretical Sciences, 
Mathematics Division, 
National Taiwan University, Taipei, 106, Taiwan.

kawamata@ms.u-tokyo.ac.jp


\begin{thebibliography}{}

\bibitem{Beilinson}
Beilinson, A. A. 
{\em Coherent sheaves on $\mathbf{P}^n$ and problems in linear algebra}. 
(Russian) Funktsional. Anal. i Prilozhen. 12 (1978), no. 3, 68--69. 

\bibitem{BN}
B\"okstedt, Marcel; Neeman, Amnon.
{\em Homotopy limits in triangulated categories}.
Compositio Math. 86 (1993), 209--234.

\bibitem{Bondal}
Bondal, A. I. 
{\em Representation of associative algebras and coherent sheaves}.
Math. USSR Izv. 34(1990), No.1, 23--42.

\bibitem{BO}
Bondal, A; Orlov, D.
{\em Semiorthogonal decompositions for algebraic varieties}.
alg-geom/9506012.

\bibitem{BvdB}
Bondal, A.; Van den Bergh, M. 
{\em Generators and representabilioty of functors in commutative and noncommutative geometry}.
Mosc. Math. J. 3 (2003), no. 1, 1--36, 258.

\bibitem{CKN}
Christensen, J. Daniel; Keller, Bernhard; Neeman, Amnon.
{\em Failure of Brown representability in derived categories}.
Topology 40 (2001), no. 6, 1339--1361. 
 
\bibitem{Kapranov}
Kapranov, M. M. 
{\em On the derived categories of coherent sheaves on some
homogeneous spaces}. 
Invent. Math. 92 (1988), no. 3, 479--508.

\bibitem{KKS}
Karmazyn, Joseph; Kuznetsov, Alexander; Shinder, Evgeny. 
{\em Derived categories of singular surfaces}.
arXiv:1809.10628.

\bibitem{DK}
Kawamata, Yujiro.
{\em D-equivalence and K-equivalence}.  
J. Diff. Geom. {\bf 61} (2002), 147--171.

\bibitem{SDG}
Kawamata, Yujiro.
{\em Birational geometry and derived categories}.
Surveys in Differential Geometry, Vol. 22, No. 1 (2017), pp. 291--317.

\bibitem{NC}
Kawamata, Yujiro.
{\em On multi-pointed non-commutative deformations and Calabi-Yau threefolds}.
Compositio Math. 154 (2018), 1815--1842. 
doi:10.1112/S0010437X18007248.

\bibitem{KMM}
Kawamata, Yujiro; Matsuda, Katsumi; Matsuki, Kenji. 
{\em Introduction to the minimal model problem}. 
in Algebraic Geometry Sendai 1985,
Advanced Studies in Pure Math. {\bf 10} (1987), 
Kinokuniya and North-Holland, 283--360. 

\bibitem{Kuznetsov}
Kuznetsov, Alexander.
{\em Derived categories of families of sextic del Pezzo surfaces}.
arXiv:1708.00522.

\bibitem{RN}
Neeman, Amnon.
{\em The connection between the K-theory localization theorem of Thomason, Trobaugh and Yao
and the smashing subcategories of Bousfield and Ravenel}.
Ann. Sci. Ecole Norm. Sup. (4) 25 (1992), no. 5, 547--566.

\bibitem{Neeman}
Neeman, Amnon.
{\em The Grothendieck duality theorem vis Bousfield's techniques and Brown representability}.
J. AMS. 9-1 (1996), 205--236.

\bibitem{Orlov1} 
Orlov, Dmitri.
{\em Triangulated categories of singularities and D-branes in Landau-Ginzburg models}.
(Russian) Tr. Mat. Inst. Steklova 246 (2004), Algebr. Geom. Metody, Svyazi i Prilozh., 240--262; 
translation in Proc. Steklov Inst. Math. 2004, no. 3(246), 227--248.
 
\bibitem{Orlov2}
Orlov, Dmitri.
{\em Formal completions and idempotent completions of triangulated categories of singularities}.
Adv. Math. 226 (2011), no. 1, 206--217. 

\bibitem{Ottaviani}
Ottaviani, Giorgio.
{\em Spinor bundles on quadrics}. 
Trans. Amer. Math. Soc. 307 (1988), no. 1, 301--316.

\bibitem{TU}
Toda, Yukinobu; Uehara, Hokuto.
{\em Tilting generators via ample line bundles}.
Adv. Math. 223 (2010), no. 1, 1--29.

\end{thebibliography}
\end{document}